\documentclass[11pt]{article}

\usepackage{graphicx}   
\usepackage{multirow}
\usepackage{amsmath}
\usepackage{amsfonts}
\usepackage{latexsym}
\usepackage{verbatim}
\usepackage{amsmath}    
\usepackage{amsmath, amssymb, algorithm, algorithmic}
\usepackage[table]{xcolor}

\usepackage{multirow}
\usepackage [pagewise,displaymath, mathlines]{lineno}
\usepackage{amssymb}
\usepackage{epstopdf}
\usepackage{amsmath}
\usepackage{epsf}
\usepackage{graphicx}
\usepackage{mathdots}
\newtheorem{theorem}{Theorem}
\newtheorem{lemma}[theorem]{Lemma}
\newtheorem{corollary}[theorem]{Corollary}

\newtheorem{remark}[theorem]{Remark}

\newcommand{\beq}{\begin{equation}}
\newcommand{\eeq}{\end{equation}}
\newcommand{\beqa}{\begin{eqnarray}}
\newcommand{\eeqa}{\end{eqnarray}}
\newcommand{\beqas}{\begin{eqnarray*}}
\newcommand{\eeqas}{\end{eqnarray*}}
\newcommand{\bi}{\begin{itemize}}
\newcommand{\ei}{\end{itemize}}
\newcommand{\ba}{\begin{array}}
\newcommand{\ea}{\end{array}}

\newcommand{\nn}{\nonumber}

\def\eqnok#1{(\ref{#1})}
\def\argmin{{\rm argmin}}

\def\Argmin{{\rm Argmin}}

\def\vgap{\vspace*{.1in}}

\newcommand{\bbr}{\Bbb{R}}

\def\cX{{\cal X}}

\def\lb{{\rm lb}}

\title{
Generalized Uniformly Optimal Methods \\
for Nonlinear Programming
\thanks{August, 2015.
This research was partially supported by NSF grants CMMI-1254446, CMMI-1537414, DMS-1319050, DMS-1016204 and ONR grant N00014-13-1-0036.}
}
\author{Saeed Ghadimi\thanks{sghadimi@ufl.edu,
Department of Industrial and Systems Engineering,
University of Florida, Gainesville, FL 32611.}
\and Guanghui Lan\thanks{glan@ise.ufl.edu,
http://www.ise.ufl.edu/glan,
Department of Industrial and Systems Engineering,
University of Florida, Gainesville, FL 32611.}
\and Hongchao Zhang\thanks{hozhang@math.lsu.edu,
https://www.math.lsu.edu/$\sim$hozhang,
Department of Mathematics, Louisiana State University,
Baton Rouge, LA 70803.}
}

\date{August 28, 2015}
\begin{document}
\maketitle
\begin{abstract}
In this paper, we present a generic framework to extend existing uniformly optimal convex programming algorithms
to solve more general nonlinear, possibly nonconvex, optimization problems.
The basic idea is to incorporate a local search step (gradient descent or Quasi-Newton iteration) into these uniformly optimal convex programming methods, and then enforce a monotone decreasing property of the function values computed along the trajectory. Algorithms of these types
will then achieve the best known complexity for nonconvex problems, and the optimal complexity for convex ones without requiring any problem parameters. As a consequence, we can have a unified treatment for a general class of nonlinear programming problems regardless of their convexity and smoothness level. In particular, we show that the accelerated gradient and level methods, both originally designed for solving convex optimization problems only, can be used for solving both convex and nonconvex problems uniformly. In a similar vein, we show that some well-studied techniques for nonlinear programming, e.g., Quasi-Newton iteration, can be embedded into optimal convex optimization algorithms to possibly further enhance their numerical performance. Our theoretical and algorithmic developments are complemented by some promising numerical results obtained for solving a few important nonconvex and nonlinear data analysis problems in the literature.
\end{abstract}

\vgap

\noindent {\bf keywords}
nonconvex optimization, uniformly optimal methods, parameter free methods, gradient descent methods,
Quasi-Newton methods, accelerated gradient methods, accelerated level methods

\noindent {\bf AMS 2000 subject classification:} 90C25, 90C06, 90C22, 49M37

\thispagestyle{plain}
\markboth{S. GHADIMI, G. LAN, AND H. ZHANG}{GENERALIZED UNIFORMLY OPTIMAL METHODS FOR NONLINEAR PROGRAMMING}

\setcounter{equation}{0}

\section{Introduction}
\label{sec_intro}
In this paper, we consider the following nonlinear programming problem
\beq \label{NLP}
\Psi^* := \min\limits_{x \in X} \{\Psi(x) := f(x) + \cX(x)\},
\eeq
where  $X \subseteq \bbr^n$ is a closed convex set, $f:X \to \bbr$ is possibly nonconvex,
and $\cX$ is a simple but possibly non-differentiable convex function with known structure
(e.g. $\cX(x) = \|x\|_1$ or $\cX$ even vanishes). Moreover, we assume that $f$ has H\"{o}lder continuous gradient on $X$ i.e.,
there exists $ H >0$ and $\nu \in [0,1]$ such that
\beq \label{f_holder1}
\|f'(y) - f'(x)\| \le H \|y-x\|^\nu \quad \forall \, x, y \in X,
\eeq
where $f'(x)$ is the gradient of $f$ at $x$ and $\|\cdot\|$ is the Euclidean norm in $\bbr^n$.
The above assumption in \eqnok{f_holder1} covers a wide range class of objective funtions,
including smooth functions (with $\nu = 1$),
weakly smooth functions (with $\nu \in (0,1)$), and nonsmooth convex functions with bounded subgradients (with $\nu = 0$).

The complexity of solving \eqnok{NLP} has been well-understood under the convex setting, i.e., when $f$ is convex.
According to the classic complexity theory by Nemirovski and Yudin~\cite{nemyud:83},
if $f$ is a general convex function with bounded subgradients (i.e., $\nu = 0$),
then the number of subgradient evaluations of $f$ required to find a solution $\bar x \in X$ such that $\Psi(\bar x) -\Psi^* \le \epsilon$
cannot be smaller than ${\cal O}(1/\epsilon^2)$ when $n$ is sufficiently large.
Here $\Psi^*$ denotes the optimal value of \eqnok{NLP}.
Such a lower complexity bound can be achieved by different first-order methods, including the subgradient and
mirror descent methods in \cite{nemyud:83}, and the bundle-level type methods in \cite{LNN,BenNem05-1}.
Moreover, if $f$ is a smooth convex function with Lipschitz continuous gradients (i.e., $\nu = 1$),
then the number of gradient evaluations of $f$  cannot be smaller than ${\cal O}(1/\sqrt{\epsilon})$ for sufficiently large $n$.
Such a lower complexity bound can be achieved by the well-known Nesterov's accelerated gradient (AG) method,
which was originally developed in~\cite{Nest83-1} for the smooth case with $\cX = 0$, and recently extended
for an emerging class of composite problems with a relatively simple nonsmooth term $\cX$~\cite{Nest13-1,BecTeb09-2,tseng08-1}.

While traditionally different
classes of convex optimization problems were solved by using different algorithms,
the last few years have seen a growing interest in the development of unified algorithms that can achieve
the optimal complexity for solving different classes of convex problems, preferably without requiring the input
of any problem parameters.
Lan showed in \cite{Lan10-3} that Nesterov's AG method in \cite{Nest83-1,Nest04} can achieve
the optimal complexity for solving not only smooth,
but also general nonsmooth optimization problems (i.e., $\nu = 0$ in  \eqnok{f_holder1})
by introducing a novel stepsize policy and some new convergence analysis techniques.
Devolder, Glineur and Nesterov \cite{DeGlNe10-1} further generalized this development and showed that
the AG method can exhibit the optimal ${\cal O}(1/\epsilon^{\frac{2}{1+3\nu}})$ complexity for solving weakly
smooth problems. These methods in \cite{Lan10-3,DeGlNe10-1} still require the input of
problem parameters like $\nu$ and $H$, and even the iteration limit $N$.
In a different research line, Lan~\cite{Lan13-1} generalized the bundle-level type methods, originally designed for nonsmooth problems, for both smooth and weakly smooth problems. He showed that these accelerated level methods are uniformly optimal for convex optimization in the sense that they can achieve the optimal complexity bounds without requiring the input of any problem parameters
and the iteration limit. Simplified variants of these methods were also proposed in \cite{CheLanOuyZha14} for solving ball-constrained and unconstrained convex optimization problems. Recently, Nesterov \cite{Nest14} presented another uniformly optimal method, namely, the universal accelerated gradient method for nonsmooth and (weakly) smooth convex programming. This method only needs the target accuracy as an input, which, similar to those in \cite{CheLanOuyZha14, Lan13-1}, completely removes the need of inputting problem parameters in \cite{Lan10-3,DeGlNe10-1}.


In general, all the above-mentioned methods require the convexity on the objective function to establish their convergence results.
When $f$ is possibly nonconvex, a different termination criterion according to the projected gradient $g_{_{X,k}}$(see e.g., \eqnok{def_proj_grad})
is often employed to analyze the complexity of the solution methods. While there is no known lower iteration complexity bound for first-order methods to solve the problem \eqnok{NLP}, the (projected) gradient-type methods \cite{Nest04,CarGouToi10-1, GhaLanZhang14} achieve the best-known iteration complexity ${\cal O}(1/\epsilon)$ to find a solution such that $\|g_{_{X,k}}\|^2 \le \epsilon$ when $f$ in \eqnok{NLP} has Lipschitz continuous gradient. Recently, Ghadimi and Lan~\cite{GhaLan15} generalized Nesterov's AG method to solve this class of nonconvex  nonlinear optimization problems. They showed that this generalized AG method
can not only achieve the best-known ${\cal O}(1/\epsilon)$ iteration complexity  for finding approximate stationary points for nonconvex problems,
but also exhibit the optimal iteration complexity if the objective function turns out to be convex. However, in oder to apply this method,
we need to assume that all the generated iterates lie in a bounded set and that the gradients of $f$ are be Lipschitz continuous, and
also requires the input of a few problem parameters a priori. Our main goal of this paper is to understand whether we can generalize some of the aforementioned uniformly optimal methods to solve a broader class of nonlinear programming given in \eqnok{NLP}, where function $f$ could be nonconvex and only weakly smooth or nonsmooth (see \eqnok{f_holder1}).
In addition to these theoretical aspects, our study has also been motivated by the following applications.
\begin{itemize}
\item In many machine learning problems, the regularized loss function in the objective is given as a summation of convex and nonconvex terms (see e.g., \cite{FanLi01-1,Mairal09, MasBaxBarFre99} ). A unified approach may help us in exploiting the possible local convex property of the objective function in this class of problems, while globally these problems are not convex.

\item Some optimization problems are given through a black-box oracle (see e.g., ~\cite{Fu02-1,AsmGlynn00,Law07}). Hence, both the smoothness level of the objective function and its convex property are unknown. A unified algorithm for both convex and nonconvex optimization and for handling the
smoothness level in the objective function automatically could achieve better convergence performance for this class of problems.
\end{itemize}

Our contribution in this paper mainly consists of the following aspects. First, we generalize Nesterov's AG method and present a unified accelerated gradient (UAG) method for solving a subclass of problem \eqnok{NLP}, where $f$ has Lipschitz continuous gradients on $X$ i.e., there exists $ L >0$ such that
\beq \label{f_smooth}
\|f'(y) - f'(x)\| \le L\|y-x\| \quad \mbox{for any } x, y \in X.
\eeq
Note that the above relation is a special case of \eqnok{f_holder1} with $\nu=1$ and $H$ replaced by $L$.
The basic idea of this method is to combine a gradient descent step with Nesterov's AG method and
maintain the monotonicity of the objective function value at the iterates generated by
the algorithm.  Hence, the UAG method contains the gradient projection and Nesterov's AG methods
as special cases (see the discussions after presenting Algorithm~\ref{alg_UAG}).
We show that this UAG method is uniformly optimal for the above-mentioned class of nonlinear programming in the sense that it achieves the best known iteration complexity (${\cal O}(1/\epsilon)$) to find at least one $k$ such that $\|g_{_{X,k}}\|^2 \le \epsilon$, and exhibits the optimal complexity (${\cal O}(1/\sqrt{\epsilon})$) to find a solution $\bar x \in X$ such that $\Psi(\bar x) -\Psi^* \le \epsilon$, if $f$ turns out to be convex. While these results had been also established in \cite{GhaLan15} for the generalization of Nestrov's AG method, the UAG method does not require the boundedness assumption
on the generated trajectory.

Second, we generalize the UAG method for solving a broader class of problems with $\nu \in [0,1]$ in \eqnok{f_holder1}, and present a unified problem-parameter free accelerated gradient (UPFAG) method. We show that this method under the convexity assumption on $f$,
similar to those in \cite{DeGlNe10-1,Lan13-1,Nest14}, achieves the optimal complexity bound
\beq\label{opt_weak_cvx}
{\cal O} \left(\left[\frac{H}{\epsilon}\right]^{\frac{2}{1+3\nu}}\right),
\eeq
it also exhibits the best-known iteration complexity
\beq\label{opt_weak_nocvx}
{\cal O} \left(\frac{H^{\frac{1}{\nu}}}{\epsilon^{\frac{1+\nu}{2\nu}}}\right).
\eeq
to reduce the squared norm of the projected gradient within $\epsilon$-accuracy for nonconvex optimization. To the best of our knowledge, this is the first time that a uniformly optimal algorithm, which does not require any problem parameter information but only takes the target accuracy and a few user-defined line search parameters as the input,  has been presented for solving smooth, nonsmooth, weakly smooth convex and nonconvex optimization.
Moreover, this algorithm can also exploit a more efficient Quasi-Newton step rather than the gradient descent step
for achieving the same iteration complexity bounds.

Third, by incorporating a gradient descent step into the framework of the bundle-level type methods, namely, the accelerated prox-level (APL) method presented in \cite{Lan13-1}, we propose a unified APL (UAPL) method for solving a class of nonlinear programming defined in \eqnok{NLP}, where $f$ satisfies \eqnok{f_holder1}. We show that this method achieves the complexity bounds in \eqnok{opt_weak_cvx} and \eqnok{opt_weak_nocvx} for both convex and nonconvex optimization implying that it is uniformly optimal for solving the aforementioned class of nonlinear programming. Moreover, we simplify this method and present its fast variant, by incorporating a gradient descent step into the framework of the fast APL method \cite{CheLanOuyZha14}, for solving ball-constrained and unconstrained problems. To the best of our knowledge, this is the first time that the bundle-level type methods are generalized for these nonconvex nonlinear programming problems.

The rest of the paper is organized as follows. In Section~\ref{sec_UAG}, we present the UAG method for solving a class of nonlinear programming problems where the objective function is the summation of a Lipschitz continuously differentiable function and a simple convex function, and establish
the convergence results. We then generalize this method in Section~\ref{sec_hol} for solving a broader class of problems where the Lipschitz
continuously differentiable function in the objective is replaced by a weakly smooth function with H\"{o}lder continuous gradient.
In Section~\ref{sec_bundle}, we provide different variants of the bundle-level type methods for solving the aforementioned class of nonlinear programming. In section~\ref{sec_num}, we show some numerical illustration of implementing the above-mentioned algorithms.

{\bf Notation.}
For a differentiable function $h: \bbr^n \to \bbr$, $h'(x)$ is the gradient of $h$ at $x$. More generally, when
$h$ is a proper convex function, $\partial h(x)$ denotes the subdifferential set of $h$ at $x$.
For $x \in \bbr^n$ and $y \in \bbr^n$, $\langle x, y \rangle$ is the standard inner product in $\bbr^n$.
The norm $\|\cdot\|$ is the Euclidean norm given by $\|x\| = \sqrt{ \langle x,x \rangle}$, while
$\|x\|_G = \sqrt{ \langle G x,x \rangle}$ for a positive definite matrix $G$.
Moreover, we let ${\cal B}(\bar x, r)$ to be the ball with radius $r$ centered at $\bar x$ i.e.,
${\cal B}(\bar x, r) =\{x \in \bbr^n \ \ | \ \ \|x-\bar x\| \le r\}$.
We denote $I$ as the identity matrix.
For any real number $r$, $\lceil r \rceil$ and $\lfloor r \rfloor$ denote the nearest integer to $r$ from above and below, respectively.
$\bbr_+$ denotes the set of nonnegative real numbers.

\setcounter{equation}{0}
\section{Unified accelerated gradient method}\label{sec_UAG}

Our goal in this section is to present a unified gradient type method to solve  problem \eqnok{NLP} when $f$ is Lipschitz continuously differentiable.
This method automatically carries the optimal theoretical convergence rate without explicitly knowing $f$ in \eqnok{NLP} is convex or not.
Compared with the optimal accelerated gradient method
presented in \cite{GhaLan15}, our algorithm does not need the uniform boundedness assumption on the iterates generated by
the algorithm. Throughout this section, we assume that the gradient of $f$ in \eqnok{NLP} is  $L$-Lipschitz continuous on $X$, i.e.,
\eqnok{f_smooth} holds, which consequently implies
\beq \label{f_smooth1}
|f(y) - f(x) - \langle f'(x), y - x \rangle | \le \frac{L}{2} \|y - x\|^2
\qquad\forall x, y \in X.
\eeq

The unified accelerated gradient (UAG) algorithm is described as follows.

\begin{algorithm} [H]
	\caption{The unified accelerated gradient (UAG) algorithm}
	\label{alg_UAG}
	\begin{algorithmic}

\STATE Input:
$x_0 \in X$, $\{\alpha_k\}$ s.t. $\alpha_1 = 1$ and $\alpha_k \in (0,1)$ for any $k \ge 2$,
$\{\lambda_k > 0\}$, $\{\beta_k > 0\}$ and $\{\gamma_k > 0\}$.
\STATE 0. Set the initial points $x^{ag}_0 = x_0$ and $k=1$.
\STATE 1. Set
\beqa
x^{md}_k &=& (1 - \alpha_k) x^{ag}_{k-1}
+ \alpha_k x_{k-1}.\label{Ne}
\eeqa
\STATE 2. Compute $f'(x^{md}_k)$ and $f'(x^{ag}_{k-1})$, and set
\beqa
x_k &=& \arg\min\limits_{u \in X} \left\{ \langle f'(x^{md}_k), u \rangle
+ \frac{1}{2 \lambda_k} \|u-x_{k-1}\|^2+\cX(u) \right\}, \label{Ne1}\\
\tilde x^{ag}_k &=& (1 - \alpha_k) x^{ag}_{k-1}+ \alpha_k x_k, \label{Ne2}\\
\bar x^{ag}_k &=& \arg\min\limits_{u \in X} \left\{ \langle f'(x^{ag}_{k-1}), u \rangle
+ \frac{1}{2 \beta_k} \|u- x^{ag}_{k-1}\|^2+\cX(u) \right\}.\label{Ne3}
\eeqa
\STATE 3. Choose $x^{ag}_k$ such that
\beq \label{def_x_ag}
\Psi(x^{ag}_k) = \min \{\Psi(\bar x^{ag}_k), \Psi(\tilde x^{ag}_k)\}.
\eeq

\STATE 4. Set $k \leftarrow k+1$ and go to step 1.
	\end{algorithmic}
\end{algorithm}

\vgap

We now add a few remarks about the above UAG algorithm. First, observe that by just considering \eqnok{Ne}, \eqnok{Ne1}, \eqnok{Ne2} and setting $x^{ag}_k = \tilde x^{ag}_k$, Algorithm~\ref{alg_UAG} would reduce to a variant of the well-known Nesterov's optimal gradient method (see, e.g., \cite{Nest04}).
Moreover, if replacing $x^{ag}_{k-1}$ by $x^{md}_k$ in \eqnok{Ne3} and setting $x^{ag}_k = \bar x^{ag}_k$, then \eqnok{Ne}, \eqnok{Ne1} and \eqnok{Ne3} would give the accelerated gradient (AG) method proposed by Ghadimi and Lan \cite{GhaLan13-1}.
However, when $f$ is nonconvex, the convergence of this AG method in \cite{GhaLan13-1} requires the boundedness assumption on the
iterates, as mentioned before. On the other hand, by just considering \eqnok{Ne3} and setting $x^{ag}_k = \bar x^{ag}_k$, the UAG algorithm would be a variant of the projected gradient method \cite{GhaLanZhang14}. Indeed, it follows from these observations that we can possibly perform the convergence analysis of the UAG method for convex and nonconvex optimization
separately (see the discussions after Corollary~\ref{lips_corl}).

Second, relation \eqnok{def_x_ag} guarantees the objective function value at the iterates $x^{ag}_k$ generated by the UAG algorithm is non-increasing.
Such a monotonicity of the objective function value, as shown in the proof of Theorem~\ref{main_theom}.a), is required to establish convergence of the algorithm when $\Psi$ is nonconvex.

Finally, noticing that $\cX$ in problem \eqnok{NLP} is not necessarily differentiable and $f$ in \eqnok{NLP} may not be a convex function, hence
we need to define a termination criterion when the objective function $\Psi$ is not convex.
In this case, we would terminate the algorithm when the norm of the generalized projected gradient defined by
\beq \label{def_proj_grad}
g_{_{X,k}} = \frac{x^{ag}_{k-1} - \bar x^{ag}_k}{\beta_k}
\eeq
is sufficiently small.
Note that $g_{_{X,k}} = f'(x^{ag}_{k-1})$ when $\cX$ vanishes and $X=\bbr^n$. Indeed the above generalized projected gradient in constrained
nonsmooth optimization plays an analogous role to that of the gradient in unconstrained smooth optimization.
 In particular, it can be shown that if $\|g_{_{X,k}}\| \le \epsilon$, then $\Psi'(\bar{x}^{ag}_k)  \in -{\cal N}_X (\bar x^{ag}_k) + {\cal B}(\epsilon (L\beta_k+1))$, where $\Psi'(\bar{x}^{ag}_k) \in \partial \Psi(\bar{x}^{ag}_k)$, ${\cal N}_X (\bar x^{ag}_k)$ is the normal cone of $X$ at $\bar x^{ag}_k$, and ${\cal B}(r) := \{x \in \bbr^n: \|x\| \le r\}$ (see e.g., \cite{GhaLan15}).

To establish the convergence of the above UAG algorithm,
we first need the following simple technical result (see
Lemma 3 of \cite{Lan13-2} for a slightly more general result).

\begin{lemma} \label{Gamma_division}
Let $\{\alpha_k\}$ be a sequence of real numbers such that  $\alpha_1 = 1$ and $\alpha_k \in (0,1)$ for any $k \ge 2$.
If a sequence $\{\omega_k\}$ satisfies
\beq \label{general_cond}
\omega_k \le (1-\alpha_k) \omega_{k-1}+\zeta_k, \ \ k=1,2,\ldots,
\eeq
then for any $k \ge 1$ we have
\[
\omega_k \le \Gamma_k \sum_{i=1}^k ( \zeta_i/\Gamma_i),
\]
where
\beq \label{def_Gamma}
\Gamma_{k} :=
\left\{
\begin{array}{ll}
 1, & k = 1,\\
(1 - \alpha_{k}) \Gamma_{k-1}, & k \ge 2.
\end{array} \right.
\eeq
\end{lemma}

Below, we give the main convergence properties of the UAG algorithm.

\begin{theorem}\label{main_theom}
Let $\{x^{ag}_k\}$ be the iterates generated by Algorithm~\ref{alg_UAG}
and $\Gamma_k$ be defined in \eqnok{def_Gamma}.
\begin{itemize}
\item [a)] Suppose that $\Psi$ is bounded below over $X$, i.e., $\Psi^*$ is finite. If $\{\beta_k\}$ is chosen such that
\beq\label{stepsize_assum0}
\beta_k \le \frac{2}{L}
\eeq
with $\beta_k < 2/L $ for at least one $k$, then for any $N \ge 1$, we have
\beq
\min\limits_{k=1,...,N} \|g_{_{X,k}}\|^2
\le \frac{\Psi(x_0) - \Psi^*}{\sum_{k=1}^N \beta_k \left(1-\frac{L \beta_k}{2}\right)}.\label{main_nocvx}
\eeq

\item [b)] Suppose that $f$ is convex and
an optimal solution $x^*$ exists for problem \eqnok{NLP}.
If $\{\alpha_k\}$, $\{\beta_k\}$ and $\{\lambda_k\}$ are chosen such that
\begin{equation}
\alpha_k \lambda_k \le \frac{1}{L}, \label{stepsize_assum1}
\end{equation}
and
\begin{equation}
\frac{\alpha_1}{\lambda_1 \Gamma_1} \ge \frac{\alpha_2}{\lambda_2 \Gamma_2} \ge \ldots,\label{stepsize_equal}
\end{equation}
then for any $N \ge 1$, we have
\beqa
\Psi(x^{ag}_N)-\Psi(x^*) &\le& \frac{\Gamma_N \|x_0 - x^*\|^2}
{2\lambda_1}. \label{cvx_fun}
\eeqa
\end{itemize}
\end{theorem}

\vgap

\begin{proof}
We first show part a).
By \eqnok{f_smooth1}, we have
\beq \label{f_smooth_p1}
f(\bar x^{ag}_k) \le f(x^{ag}_{k-1}) + \langle f'( x^{ag}_{k-1}) , \bar x^{ag}_k- x^{ag}_{k-1} \rangle +\frac{L}{2} \|\bar x^{ag}_k- x^{ag}_{k-1} \|^2.
\eeq
Then, it follows from  \eqnok{Ne3} and  Lemma 2 of \cite{GhaLan12-2a} that for any $x  \in X$ we have
\beq \label{Ne3_lemma}
\langle f'(x^{ag}_{k-1}), \bar x^{ag}_k - x \rangle + \cX(\bar x^{ag}_k)  \le \cX(x) +
\frac{1}{2\beta_k} \left[\|x^{ag}_{k-1} - x\|^2-\|\bar x^{ag}_k - x\|^2 -\|x^{ag}_{k-1} -\bar x^{ag}_k\|^2 \right].
\eeq
 Letting $x=x^{ag}_{k-1}$ in the above inequality, adding it to \eqnok{f_smooth_p1} and noticing \eqnok{def_x_ag},
we obtain
\beq \label{cnvg_nocvx}
\Psi(x^{ag}_k) \le \Psi(\bar x^{ag}_k)\le \Psi(x^{ag}_{k-1}) -\frac{1}{\beta_k} \left(1-\frac{L \beta_k}{2}\right) \| x^{ag}_{k-1} -\bar x^{ag}_k\|^2.
\eeq
Summing up the above inequalities for $k$ from $1$ to $N$ and rearranging the terms, it follows from the definition
of $g_{_{X,k}}$ in \eqnok{def_proj_grad} and  $\Psi^* \le \Psi(x^{ag}_N)$ that
\beqa
&& \min\limits_{k=1,2,\ldots,N} \|g_{_{X,k}}\|^2 \sum_{k=1}^N \beta_k \left(1-\frac{L \beta_k}{2}\right) \nn \\
& \le & \sum_{k=1}^N \beta_k \left(1-\frac{L \beta_k}{2}\right) \|g_{_{X,k}}\|^2
=  \sum_{k=1}^N \frac{1}{\beta_k} \left(1-\frac{L \beta_k}{2}\right) \|x^{ag}_{k-1} -\bar x^{ag}_k\|^2 \nn \\
&\le& \Psi(x^{ag}_0)-\Psi(x^{ag}_N) \le \Psi(x_0)-\Psi^*.
\eeqa
Then,  \eqnok{stepsize_assum0} and  the above inequality  imply \eqnok{main_nocvx}.

We now show part b). Using \eqnok{Ne2} and the convexity of $f$, for any $x \in X$, we have
\beqa
& &f(x^{md}_k) + \langle f'(x^{md}_k) , \tilde x^{ag}_k-x^{md}_k \rangle \nn \\
&=& (1-\alpha_k)[f(x^{md}_k) + \langle f'(x^{md}_k) , x^{ag}_{k-1}-x^{md}_k \rangle] + \alpha_k [f(x^{md}_k) + \langle f'(x^{md}_k) , x_k-x^{md}_k \rangle] \nn \\
&\le& (1-\alpha_k)f(x^{ag}_{k-1})+\alpha_k f(x) +\alpha_k \langle f'(x^{md}_k) , x_k-x\rangle,
\eeqa
which together with \eqnok{f_smooth1}, \eqnok{Ne}, \eqnok{Ne2}, \eqnok{def_x_ag}, and the convexity of
$\cX$ imply that
\begin{align}
&\Psi (x^{ag}_k) \le \Psi (\tilde x^{ag}_k) =
f(\tilde x^{ag}_k)+\cX(\tilde x^{ag}_k) \le f(x^{md}_k) + \langle f'(x^{md}_k) , \tilde x^{ag}_k-x^{md}_k \rangle +\frac{L}{2} \|\tilde x^{ag}_k-x^{md}_k\|^2+\cX(\tilde x^{ag}_k) \nn \\
&\le (1-\alpha_k)f(x^{ag}_{k-1})+\alpha_k f(x) +\alpha_k \langle f'(x^{md}_k) , x_k-x\rangle
+\frac{L}{2} \|\tilde x^{ag}_k-x^{md}_k\|^2 + (1-\alpha_k) \cX(x^{ag}_{k-1})+\alpha_k \cX(x_k) \nn \\
&= (1-\alpha_k)\Psi(x^{ag}_{k-1})+\alpha_k f(x) +\alpha_k \langle f'(x^{md}_k) , x_k-x\rangle
+\frac{L\alpha_k^2}{2} \|x_k-x_{k-1}\|^2+\alpha_k \cX(x_k).\label{f_smooth_p2}
\end{align}
Now, by \eqnok{Ne1} and Lemma 2 of \cite{GhaLan12-2a}, for any $x  \in X$ we have
\beq \label{Ne1_lemma}
\langle f'(x^{md}_k), x_k - x \rangle + \cX(x_k)  \le \cX(x) +
\frac{1}{2\lambda_k} \left[\|x_{k-1} - x\|^2-\|x_k - x\|^2 -\|x_k -x_{k-1}\|^2 \right].
\eeq
Multiplying the above inequality by $\alpha_k$ and summing it up with \eqnok{f_smooth_p2}, we obtain
\[
\Psi (x^{ag}_k) \le (1-\alpha_k)\Psi(x^{ag}_{k-1})+\alpha_k \Psi(x)+\frac{\alpha_k}{2\lambda_k} \left[\|x_{k-1} - x\|^2-\|x_k - x\|^2\right]
-\frac{\alpha_k(1-L\alpha_k\lambda_k)}{2\lambda_k}\|x_k -x_{k-1}\|^2,
\]
which together with the assumption \eqnok{stepsize_assum1} give
\beq \label{cnvg_cvx}
\Psi (x^{ag}_k) \le (1-\alpha_k)\Psi(x^{ag}_{k-1})+\alpha_k \Psi(x)+\frac{\alpha_k}{2\lambda_k} \left[\|x_{k-1} - x\|^2-\|x_k - x\|^2\right].
\eeq
 Subtracting $\Psi(x)$ from both sides of the above inequality and dividing them by $\Gamma_k$, then it follows
from Lemma~\ref{Gamma_division} that for any $x \in X$ we have
\beq
\frac{\Psi(x^{ag}_N)-\Psi(x)}{\Gamma_N} \le \sum_{k=1}^N \frac{\alpha_k}{2 \lambda_k \Gamma_k} \left[\|x_{k-1}-x\|^2 -\|x_k-x\|^2 \right]. \label{cvx_recur}
\eeq
Now, by \eqnok{stepsize_equal} and $\alpha_1 = \Gamma_1 = 1$, we have
\beq \label{sum_dist}
\sum_{k=1}^N \frac{\alpha_k}{\lambda_k \Gamma_k} \left[\|x_{k-1}-x\|^2 -\|x_k-x\|^2 \right]
\le \frac{\alpha_1 \|x_0 - x\|^2}{\lambda_1 \Gamma_1} = \frac{\|x_0 - x\|^2}{\lambda_1},
\eeq
which together with \eqnok{cvx_recur} give
\[
\frac{\Psi(x^{ag}_N)-\Psi(x)}{\Gamma_N} \le \frac{\|x_0 - x\|^2}{2 \lambda_1}.
\]
Setting  $x=x^*$ in the above inequality directly gives \eqnok{cvx_fun}.
\end{proof}

\vgap

Note that the convergence analysis of the UAG method is completely separable for convex and nonconvex problems
in the above proof.
This allows us to solve the optimization problems uniformly without having any information about the convexity of $f$. In the next corollary we specify one particular set of choices of  the parameters in the UAG algorithm to obtain a particular convergence rate.

\begin{corollary} \label{lips_corl}
Suppose that $\{\alpha_k\}$, $\{\beta_k\}$ and $\{\lambda_k\}$ in Algorithm~\ref{alg_UAG} are set to
\beq \label{def_alpha_lambda_beta0}
\alpha_k = \frac{2}{k+1}, \ \ \ \lambda_k = \frac{k}{2L},  \ \ \ \mbox{and} \ \ \ \beta_k = \frac{1}{L}.
\eeq

\begin{itemize}
\item [a)]  Suppose that $\Psi$ is bounded below over $X$, i.e., $\Psi^*$ is finite. Then for any $N \ge 1$, we have
\beq
\min\limits_{k=1,...,N} \|g_{_{X,k}}\|^2
\le \frac{2L[\Psi(x_0) - \Psi^*]}{N}.\label{main_nocvx1}
\eeq

\item [b)] Suppose that $f$ is convex and
an optimal solution $x^*$ exists for problem \eqnok{NLP}. Then for any $N \ge 1$, we have
\beqa
\Psi(x^{ag}_N)-\Psi(x^*) &\le& \frac{2 L \|x_0 - x^*\|^2}
{N(N+1)}. \label{cvx_fun1}
\eeqa
\end{itemize}

\end{corollary}

\begin{proof}
We first show part a). Observe that by \eqnok{def_alpha_lambda_beta0} condition \eqnok{stepsize_assum0} holds.
Then, \eqnok{main_nocvx1} directly follows from \eqnok{main_nocvx}
and
\[
\sum_{k=1}^N \beta_k \left(1-\frac{L \beta_k}{2}\right)= \frac{N}{2L}.
\]

We now show part b). Observe that by \eqnok{def_alpha_lambda_beta0}, we have
\begin{align}
\alpha_k \lambda_k &= \frac{k}{(k+1)L} <= \frac{1}{L} \quad \mbox{ and } \nn  \\
\frac{\alpha_1}{\lambda_1 \Gamma_1} &= \frac{\alpha_2}{\lambda_2 \Gamma_2} = \ldots =2 L,\nn
\end{align}
which imply that conditions \eqnok{stepsize_assum1} and \eqnok{stepsize_equal} hold.
On the other hand, by \eqnok{def_Gamma} and \eqnok{def_alpha_lambda_beta0},
we have
\[
\Gamma_N = \frac{2}{N(N+1)},
\]
which together with \eqnok{cvx_fun} clearly imply \eqnok{cvx_fun1}.
\end{proof}

\vgap

We now add a few remarks about the results obtained in Corollary~\ref{lips_corl}. First, note that the UAG method achieves the best known convergence rate for solving nonconvex optimization problems as well as convex optimization problems.
Specifically, \eqnok{main_nocvx1} implies that to achieve $\|g_{_{X,k}}\|^2 \le \epsilon$ for at least one $k$, the total number of iterations needed by the UAG method is bounded by
\beq \label{main_nocvx2}
{\cal O} \left(\frac{L[\Psi(x_0) - \Psi^*]}{\epsilon}\right).
\eeq
This bound is also known for the steepest descent method for unconstrained problems \cite{Nest04}, and the projected gradient method for composite problems in Ghadimi, Lan and Zhang \cite{GhaLanZhang14}.
A similar bound is also obtained by the AG method \cite{GhaLan13-1}, which however, for composite problems, relies on an additional assumption
that the iterates are bounded as mentioned before. On the other hand, one possible advantage of this AG method in \cite{GhaLan13-1} exists in that it can separate the affects of the Lipschitz constants of smooth convex terms.
When $f$ is convex, by \eqnok{cvx_fun1}, the UAG method guarantees to find a solution $\bar x$ such that
 $\Psi(\bar x) - \Psi(x^*) \le \epsilon$ in at most
\beq \label{cvx_fun2}
{\cal O} \left(\sqrt{\frac{L}{\epsilon}}\|x_0 - x^*\|\right)
\eeq
iterations, which is known to be optimal for solving convex optimization problems \cite{nemyud:83}.

Second, the UAG method does not need to know the convexity of the objective function as a prior knowledge.
Instead, it treats both the convex and nonconvex optimization problems in a unified way.
In any case, the UAG method always achieves the complexity bound in \eqnok{main_nocvx2}.
And when the objective function happens to be convex, it would also achieve the optimal complexity bound in \eqnok{cvx_fun2}.

Despite the above mentioned theoretical advantages for the UAG method, there are still some practical drawbacks of this
method. One obvious drawback of the UAG method is that the parameter policy in \eqnok{def_alpha_lambda_beta0} requires the knowledge of the Lipschitz constant $L$ which may not be exactly known in practice. And a poor estimate of this Lipschitz constant
may severely deteriorate the performance of the method \cite{NJLS09-1}. On the other hand, in many applications the Lipschitz continuity
of the gradient of $f$ in \eqnok{NLP} may not be known either. And in fact, the gradient of $f$ may be only h\"{o}lder continuous instead of Lipschitz continuous. Furthermore, we can see from the convergence analysis of the case when $f$ is not a convex function that the UAG method will perform more like the steepest descent method. However,
the steepest descent method, although very robust, is usually not an efficient method as verified in many applications of nonlinear programming.
In the next section, we would like to modify the UAG method to a more practical method so that it could
be applied to a broader class of problems and be more flexible to take advantage of some already well-studied efficient methods in nonlinear programming. In addition,  no prior knowledge of problem parameters is needed.

\setcounter{equation}{0}
\section{Unified problem-parameter free accelerated gradient method}\label{sec_hol}

In this section, we consider a broader class of problems including smooth, weakly smooth and nonsmooth objective functions. In particular, we would
like to deal
with the class of problems in \eqnok{NLP} such that the gradient of $f$ is H\"{o}lder continuous in the sense of \eqnok{f_holder1}, which also implies
\beq \label{f_holder2}
|f(y) - f(x) - \langle f'(x), y - x \rangle |
\le \frac{H}{1+\nu} \|y - x\|^{1+\nu} \ \ \
\mbox{ for any } x, y  \in X.
\eeq
Our algorithm is stated below as Algorithm~\ref{alg_UPFAG} which involves two line search procedures. For this algorithm, we have the following remarks.

First, note that in steps 1 and 2, we implement two independent line search procedures, respectively, in \eqnok{line_search1} and \eqnok{line_search2}. Indeed, we start with initial choices of stepsizes $\hat{\lambda}_k$ and $\hat{\beta}_k$, and then perform Armijo type of line searches such that certain specially designed line search conditions are satisfied.
We will show in Theorem~\ref{main_theom_hol}.a) that
the two line search procedures will finish in finite number of inner iterations. One simple choice of line search in practice is to set
$G_k = I$ all $k \ge 1$, and set the initial stepsizes to be some Barzilai-Borwein type stepsizes  such as
\beq\label{initial-stepsize}
\hat{\lambda}_k = \max \left\{ \frac{\langle s^{md}_{k-1}, y^{md}_{k-1} \rangle}{\langle y^{md}_{k-1}, y^{md}_{k-1} \rangle}, \sigma  \right\}  \mbox{ for } k \ge 1, \quad \mbox{and} \quad
\hat{\beta}_k = \max \left\{ \frac{\langle s^{ag}_{k-1}, y^{ag}_{k-1} \rangle}{\langle y^{ag}_{k-1}, y^{ag}_{k-1} \rangle}, \sigma  \right\} \mbox{ for } k > 1,
\eeq
 where $s^{md}_{k-1} = x^{md}_k - x^{ag}_{k-1}$, $y^{md}_{k-1} = f'(x^{md}_k) - f'(x^{ag}_{k-1})$,  $s^{ag}_{k-1} = x^{ag}_{k-1} - x^{ag}_{k-2}$ and
$y^{ag}_{k-1} = f'(x^{ag}_{k-1}) - f'(x^{ag}_{k-2})$. And we can choose
$\hat{\beta}_1=1/\hat H$, where $\hat H$ is an estimation of  the H\"{o}lder continuity constant in \eqnok{f_holder1}.

Second, we can include some curvature information of $f$ into a positive definite matrix $G_k$ in \eqnok{subproblem2}
to have better local approximation of the function $\Psi$ at $x^{ag}_{k-1}$. In this case, unit initial stepsize is often preferred,
that is to set $\hat{\beta}_k = 1$ for all $k \ge 1$.
In practice, we can set $G_k$ to be some Quasi-Newton matrix, e.g., the well-known BFGS or limited memory BFGS matrix (see e.g., \cite{ByNoSc94, Noc80, NocWri99}).
Then, when $X = \bbr^n $ and $\cX(x) \equiv 0$, \eqnok{subproblem2} will be exactly
a Quasi-Newton step and hence, a fast local convergence rate could be expected in practice.
When  $X \ne \bbr^n $ or $\cX(x) \not\equiv 0$, we may not have closed formula for the solution of the subproblem
\eqnok{subproblem2}. Then, the alternating direction method of multipliers or primal-dual type algorithms could solve the subproblem
\eqnok{subproblem2} quite efficiently, since its objective function is just a composition of  a simple convex function  $\cX$
and a convex quadratic function with known inverse of the Hessian. So, in general, by different choices of the matrix $G_k$,
many well-studied efficient methods in nonlinear programming could be incorporated into the algorithm.

\begin{algorithm} [H]
	\caption{The unified problem-parameter free accelerated gradient (UPFAG) algorithm}
	\label{alg_UPFAG}
	\begin{algorithmic}

\STATE Input:
$x_0 \in X$, line search parameters $0 < \gamma < \sigma < 1, \gamma_1, \gamma_2 \in (0,1)$, and
accuracy parameter $\delta>0$.
\STATE 0. Set the initial points $x^{ag}_0 = x_0$ and $k=1$.
\STATE 1. Choose initial stepsize $\hat{\lambda}_k >0$ and  find the smallest integer $\tau_{1,k} \ge 0$ such that with
\beq \label{def_alpha_lambda}
\eta_k = \hat{\lambda}_k \gamma_1^{\tau_{1,k}} \quad \mbox{ and } \quad \lambda_k = (\eta_k + \sqrt{\eta_k^2 + 4 \eta_k \Lambda_{k-1}})/2,
\eeq
the solutions obtained by \eqnok{Ne}, \eqnok{Ne1} and \eqnok{Ne2} satisfy
\beq \label{line_search1}
f(\tilde x^{ag}_k) \le f(x^{md}_k) + \alpha_k \langle f'(x^{md}_k), x_k -x_{k-1} \rangle+\frac{ \alpha_k}{2 \lambda_k}
 \|x_k -x_{k-1}\|^2+\delta \alpha_k,
\eeq
where
\beq \label{def_alpha}
\alpha_k = \lambda_k/\Lambda_k \ \ \mbox{and} \ \ \Lambda_k = \sum_{i=1}^k \lambda_k.
\eeq
\STATE 2.   Choose initial stepsize $\hat{\beta}_k >0$ and find the smallest integer $\tau_{2,k} \ge 0$ such that with
\beq \label{def_beta}
\beta_k = \hat{\beta}_k \gamma_2^{\tau_{2,k}},
\eeq
we have
\beq \label{line_search2}
\Psi(\bar x^{ag}_k) \le \Psi(x^{ag}_{k-1}) - \frac{\gamma}{2\beta_k}  \|\bar x^{ag}_k- x^{ag}_{k-1}\|^2+\frac{1}{k},
\eeq
where
\beq\label{subproblem2}
\bar x^{ag}_k = \arg\min\limits_{u \in X} \left\{ \langle f'(x^{ag}_{k-1}), u \rangle +
  \frac{1}{2 \beta_k} \|u- x^{ag}_{k-1}\|_{G_k}^2
+\cX(u) \right\},
\eeq
and $G_k \succeq \sigma I$ for some $\sigma \in (0,1)$.
\STATE 3. Choose $x^{ag}_k$ such that
\beq
\Psi(x^{ag}_k) = \min \{\Psi(x^{ag}_{k-1}), \Psi(\bar x^{ag}_k), \Psi(\tilde x^{ag}_k)\}. \label{def_x_ag_ls}
\eeq

\STATE 4. Set $k \leftarrow k+1$ and go to step 1.
	\end{algorithmic}
\end{algorithm}

\vgap

Third, from the complexity point of view, instead of setting the initial stepsizes given in \eqnok{initial-stepsize},
we could also take advantage of the line search in the previous iteration and set
\beq \label{def_alpha_lambda_beta}
\hat{\lambda}_k = \eta_{k-1} \quad \mbox{ and } \quad \hat{\beta}_k = \beta_{k-1},
\eeq
where $\eta_{k-1}$ and $\beta_{k-1}$ are the accepted stepsizes in the previous $k-1$-th iteration.
The choice of initial stepsizes in \eqnok{initial-stepsize} is more aggressive and inherits some quasi-Newton information,
and hence, could perform better in practice. However, the strategies in
\eqnok{def_alpha_lambda_beta} would have theoretical advantages in the total number of inner iterations needed in the line search
(see the discussion after Corollary~\ref{main_corl_hol}). Furthermore, notice that the choice of $\alpha_k$ can be different than the
one in \eqnok{def_alpha}. In fact, we only need to choose $\alpha_k$ such that condition \eqnok{stepsize_assum1} is satisfied.
For simplicity, we use the choice of $\alpha_k$ in \eqnok{line_search1}, which would satisfy the condition $\lambda_1 \alpha_k = \lambda_k \Gamma_k$
due to the definition of $\Gamma_k$ in \eqnok{def_Gamma}.
We can easily see that by this choice of $\alpha_k$ we always have $\alpha_1=1$ and $\alpha_k \in (0,1)$ for all $k \ge 2$.

Finally, \eqnok{def_x_ag_ls} has one more extra term $\Psi(x^{ag}_{k-1})$ in the minimum argument than \eqnok{def_x_ag}.
This extra term is designed to guarantee that $\Psi(x^{ag}_k)$ is monotonically non-increasing. Note that since we assume the stepsizes in Algorithm~\ref{alg_UAG} are set in advance, we do not need such an extra term.
In particular, $\Psi(x^{ag}_k)$ in Algorithm~\ref{alg_UAG} is non-increasing due to \eqnok{cnvg_nocvx}.

Below, we present the main convergence properties of the UPFAG algorithm.

\begin{theorem}\label{main_theom_hol}
Let $\{x^{ag}_k\}$ be the iterates generated by Algorithm~\ref{alg_UPFAG} and $\Gamma_k$ be defined in \eqnok{def_Gamma}.
\begin{itemize}
\item [a)] The line search procedures in Step 1 and Step 2 of the algorithm will finish in finite number of inner iterations.

\item [b)]  Suppose that $f$ is bounded below over $X$, i.e., $\Psi^*$ is finite. Then, for any $N \ge 1$, we have
\beq
\min\limits_{k=1,...,N} \|\hat{g}_{_{X,k}}\|^2
\le \frac{2\left[\Psi(x_0) - \Psi^*+\sum_{k=\lfloor N/2 \rfloor+1}^N k^{-1}\right]}{\gamma \sum_{k=\lfloor N/2 \rfloor+1}^N \beta_k},\label{main_nocvx_hol}
\eeq
where $\hat{g}_{_{X,k}} = (x^{ag}_{k-1} - \bar x^{ag}_k)/\beta_k$ and $\bar x^{ag}_k$ is the solution of \eqnok{subproblem2}.
\item [c)] Suppose that $\Psi$ is convex and an optimal solution $x^*$ exists for problem \eqnok{NLP}. Then for any $N \ge 1$, we have
\beqa
\Psi(x^{ag}_N)-\Psi(x^*) &\le& \frac{\Gamma_N\|x_0 - x^*\|^2}
{2\lambda_1}+\delta. \label{cvx_fun_hol}
\eeqa
\end{itemize}

\end{theorem}

\vgap

\begin{proof}
We first show part a). By \eqnok{f_holder2}, we have
\beq \label{f_holder_p1}
f(\bar x^{ag}_k) \le f(x^{ag}_{k-1}) + \langle f'(x^{ag}_{k-1}) , \bar x^{ag}_k- x^{ag}_{k-1} \rangle +\frac{H}{1+\nu} \|\bar x^{ag}_k- x^{ag}_{k-1} \|^{1+\nu}.
\eeq
Analogous to \eqnok{Ne3_lemma}, by \eqnok{line_search2} and \eqnok{subproblem2}, we can show
\[
\langle f'(x^{ag}_{k-1}), \bar x^{ag}_k - x \rangle + \cX(\bar x^{ag}_k)  \le \cX(x) +
\frac{1}{2\beta_k} \left[\|x^{ag}_{k-1} - x\|_{G_k}^2-\|\bar x^{ag}_k - x\|_{G_k}^2 -\|x^{ag}_{k-1} -\bar x^{ag}_k\|_{G_k}^2 \right].
\]
Letting $u=x^{ag}_{k-1}$ in the above inequality and summing it up with \eqnok{f_holder_p1}, we have
\beq \label{cnvg_nocvx_hol}
\Psi(\bar x^{ag}_k) \le \Psi(x^{ag}_{k-1}) - \left(\frac{\|\bar x^{ag}_k- x^{ag}_{k-1} \|_{G_k}^2}{\beta_k}-\frac{H\|\bar x^{ag}_k- x^{ag}_{k-1} \|^{1+\nu}}{1+\nu} \right).
\eeq
Now, for $\nu \in [0,1)$, it follows from the inequality $ab \le a^p/p+b^q/q$ with $p=\frac{2}{1+\nu}$, $q =  \frac{2}{1-\nu} $, and
\[
a = \frac{H}{1+\nu} \left[\frac{(1-\nu)k}{2}\right]^{\frac{1-\nu}{2}}\|x^{ag}_k-\bar x^{ag}_k\|^{1+\nu} \quad \mbox{ and } \quad
b = \left[\frac{2}{(1-\nu)k}\right]^{\frac{1-\nu}{2}}
\]
that
\beq \label{holder_inq}
\frac{H}{1+\nu} \|x^{ag}_k-\bar x^{ag}_k\|^{1+\nu}
=  a b \le L(\nu,H)k^{\frac{1-\nu}{1+\nu}} \|x^{ag}_k-\bar x^{ag}_k\|^2 +\frac{1}{k},
\eeq
where
\beq \label{def_L}
L(\nu,H) = \left\{\frac{H}{2\left[\frac{(1+\nu)}{1-\nu}\right]^{\frac{1-\nu}{2}}}\right\}^{\frac{2}{1+\nu}}.
\eeq
Let us define
\beq\label{L1H}
L(1,H) = \lim_{\nu \to 1} L(\nu,H) = \frac{H}{2}.
\eeq
Then, \eqnok{holder_inq} holds for all $\nu \in [0,1]$.
Combining \eqnok{cnvg_nocvx_hol} and \eqnok{holder_inq}, we have from $G_k \succeq \sigma I$ that
\beq \label{main_rec_hold}
\Psi(\bar x^{ag}_k) \le \Psi(x^{ag}_{k-1}) - \frac{\sigma- L(\nu,H) k^{\frac{1-\nu}{1+\nu}} \beta_k}{\beta_k} \|\bar x^{ag}_k- x^{ag}_{k-1}\|^2
+ \frac{1}{k}.
\eeq
Also, by \eqnok{f_holder2}, \eqnok{Ne}, and \eqnok{Ne2}, we have
\beqa
f(\tilde x^{ag}_k) &\le& f(x^{md}_k) + \langle f'(x^{md}_k) , \tilde x^{ag}_k- x^{md}_k \rangle +\frac{H}{1+\nu} \|\tilde x^{ag}_k- x^{md}_k \|^{1+\nu} \nn \\
&=&f(x^{md}_k) + \alpha_k \langle f'(x^{md}_k) , x_k- x_{k-1} \rangle +\frac{H \alpha_k^{1+\nu}}{1+\nu} \|x_k- x_{k-1}\|^{1+\nu}\nn \\
&=&f(x^{md}_k) + \alpha_k \langle f'(x^{md}_k) , x_k- x_{k-1} \rangle
-\alpha_k \left(\frac{\|x_k- x_{k-1}\|^2}{2\lambda_k}-\frac{H \alpha_k^{\nu}}{1+\nu} \|x_k- x_{k-1}\|^{1+\nu}\right)\nn \\
&& +\frac{\alpha_k}{2\lambda_k} \|x_k- x_{k-1}\|^2 \nn \\
&\le&f(x^{md}_k) + \alpha_k \langle f'(x^{md}_k) , x_k- x_{k-1} \rangle
-\frac{\alpha_k\left(1-2L(\nu,H)\alpha_k^{\frac{2\nu}{1+\nu}}\lambda_k \delta^{\frac{\nu-1}{1+\nu}}\right)\|x_k- x_{k-1}\|^2}{2\lambda_k}\nn \\
&& + \frac{\alpha_k}{2\lambda_k}\|x_k- x_{k-1}\|^2+\delta \alpha_k,\label{main_rec_hold2}
\eeqa
where the last inequality is obtained similar to \eqnok{holder_inq} and $L(\nu,H)$ is defined in \eqnok{def_L} and
\eqnok{L1H}.

Now, observe that if
\beq \label{line_search_cond1}
\alpha_k^{\frac{2\nu}{1+\nu}} \lambda_k \le \frac{\delta^{\frac{1-\nu}{1+\nu}}}{2 L(\nu,H)} \quad \mbox{ and } \quad
\beta_k \le \frac{(2\sigma - \gamma) k^{\frac{\nu-1}{1+\nu}}}{2 L(\nu,H)},
\eeq
then \eqnok{main_rec_hold} and \eqnok{main_rec_hold2}, respectively, imply
\eqnok{line_search2} and \eqnok{line_search1}.
By \eqnok{def_alpha_lambda} and our setting of $\alpha_k = \lambda_k/\Lambda_k = \lambda_k/(\lambda_k + \Lambda_{k-1})$, we have $\alpha_k \lambda_k = \eta_k$. Hence,
\beq\label{alphabeta}
\alpha_k^{\frac{2\nu}{1+\nu}} \lambda_k = (\alpha_k \lambda_k )^{\frac{2\nu}{1+\nu}} \lambda_k^{\frac{1-\nu}{1+\nu}}
= \eta_k^{\frac{2\nu}{1+\nu}} \lambda_k^{\frac{1-\nu}{1+\nu}}.
\eeq
By \eqnok{def_alpha_lambda}, \eqnok{def_beta} and $\gamma_1, \gamma_2 \in (0,1)$, we have
\[
\lim_{\tau_{1,k} \to \infty} \eta_k = 0, \quad \lim_{\eta_k  \to 0} \lambda_k = 0 \quad \mbox{and}
\quad \lim_{\tau_{2,k} \to \infty} \beta_k = 0,
\]
for any fixed $k$, which together with \eqnok{alphabeta}
 imply \eqnok{line_search_cond1} will be finally satisfied in the line search procedure and therefore, \eqnok{line_search2}
and \eqnok{line_search1} will essentially be satisfied. So the line search procedures in
Step 1 and Step 2  of Algorithm~\ref{alg_UPFAG} are well-defined and finite.

We now show part b). Noting \eqnok{line_search2}, \eqnok{def_x_ag_ls}, and in view of \eqnok{def_proj_grad}, we have
\[
\Psi(x^{ag}_k) \le \Psi(\bar x^{ag}_k) \le \Psi(x^{ag}_{k-1}) - \frac{ \gamma \|\bar x^{ag}_k- x^{ag}_{k-1}\|^2}{2\beta_k}
+\frac{1}{k} = \Psi(x^{ag}_{k-1}) - \frac{\gamma \beta_k}{2} \|\hat{g}_{_{X,k}}\|^2
+\frac{1}{k}.
\]
Summing up the above inequalities for $k$ from $\lfloor N/2 \rfloor+1$ to $N$ and re-arranging the terms, we obtain
\begin{align}
&\min\limits_{k=1,2,\ldots,N} \|\hat{g}_{_{X,k}}\|^2 \sum_{k=\lfloor N/2 \rfloor+1}^N \frac{\gamma \beta_k}{2}
\le \min\limits_{k=\lfloor N/2 \rfloor+1,2,\ldots,N} \|\hat{g}_{_{X,k}}\|^2 \sum_{k=\lfloor N/2 \rfloor+1}^N \frac{\gamma \beta_k}{2}
\le \sum_{k=\lfloor N/2 \rfloor+1}^N \frac{ \gamma \beta_k}{2}\|\hat{g}_{_{X,k}}\|^2 \nn \\
&\le \Psi(x^{ag}_{\lfloor N/2 \rfloor}) -\Psi(x^{ag}_N)+\sum_{k=\lfloor N/2 \rfloor+1}^N \frac{1}{k} \le \Psi(x_0) -\Psi(x^*)+\sum_{k=\lfloor N/2 \rfloor+1}^N \frac{1}{k},
\end{align}
where the last inequality follows from \eqnok{def_x_ag_ls} and the fact that $\Psi^* \le \Psi(x^{ag}_N)$.
Dividing both sides of the above inequality by $\gamma \sum_{k=\lfloor N/2 \rfloor+1}^N \frac{\beta_k}{2}$, we clearly obtain \eqnok{main_nocvx_hol}.

\vgap

We now show part c). By \eqnok{Ne}, and \eqnok{Ne2}, and \eqnok{line_search1}, we have
\beqa
f(\tilde x^{ag}_k)&\le&f(x^{md}_k) + \alpha_k \langle f'(x^{md}_k) , x_k- x_{k-1} \rangle+\frac{\alpha_k\|x_k- x_{k-1}\|^2}{2\lambda_k}+\delta \alpha_k \nn \\
&=& f(x^{md}_k)+\langle f'(x^{md}_k) , \tilde x^{ag}_k- x^{md}_k \rangle+\frac{\alpha_k\|x_k- x_{k-1}\|^2}{2\lambda_k}+\delta \alpha_k.
\eeqa
Combining the above inequality with \eqnok{def_x_ag_ls} and noticing the convexity of $\cX$, similar to \eqnok{f_smooth_p2},
for any $x \in X$ we have
\[
\Psi (x^{ag}_k) \le
(1-\alpha_k)\Psi(x^{ag}_{k-1})+\alpha_k f(x) +\alpha_k \langle f'(x^{md}_k) , x_k-x\rangle
+\frac{\alpha_k }{2\lambda_k} \|x_k-x_{k-1}\|^2+\alpha_k \cX(x_k)+ \delta \alpha_k.
\]
Adding the above inequality to \eqnok{Ne1_lemma} (with its both sides multiplied by $\alpha_k$), we have
\beq \label{UPFAG_rec}
\Psi (x^{ag}_k) \le (1-\alpha_k)\Psi(x^{ag}_{k-1})+\alpha_k \Psi(x)+\frac{\alpha_k}{2\lambda_k} \left[\|x_{k-1} - x\|^2-\|x_k - x\|^2\right]+ \delta \alpha_k.
\eeq
Also, note that by \eqnok{def_alpha} and \eqnok{def_Gamma}, we can easily show that
\beq \label{def_Gamma2}
\Gamma_k = \frac{\lambda_1}{\sum_{i=1}^k \lambda_i} \ \  \mbox{and} \ \ \frac{\alpha_k}{\lambda_k \Gamma_k} = \frac{1}{\lambda_1} \ \ \ \forall k \ge 1.
\eeq
Subtracting $\Psi(x)$ from both sides of \eqnok{UPFAG_rec}, dividing them by $\Gamma_k$, then it follows from
Lemma~\ref{Gamma_division} that for any $x \in X$ we have
\beqa
\frac{\Psi(x^{ag}_N)-\Psi(x)}{\Gamma_N} &\le& \sum_{k=1}^N \frac{\alpha_k}{2 \lambda_k \Gamma_k} \left[\|x_{k-1}-x\|^2 -\|x_k-x\|^2 \right]
+\delta\sum_{k=1}^N \frac {\alpha_k}{\Gamma_k}\nn \\
&\le&\frac{\|x_0 - x\|^2}{2 \lambda_1}
+\frac{\delta}{\Gamma_N},\label{cvx_recur_hol}
\eeqa
where the second inequality follows from \eqnok{def_Gamma2} and the fact that
\beq \label{sum_Gamma}
\sum_{k = 1}^N \frac{\alpha_k}{\Gamma_k}
= \frac{\alpha_1}{\Gamma_1} +
\sum_{k = 2}^k \frac{1}{\Gamma_k} \left(1 - \frac{\Gamma_k}
{\Gamma_{k-1}} \right)
= \frac{1}{\Gamma_1} + \sum_{k = 2}^k \left(\frac{1}
{\Gamma_k} - \frac{1}{\Gamma_{k-1}} \right)
= \frac{1}{\Gamma_N}.
\eeq
Then, \eqnok{cvx_fun_hol} follows immediately  from
\eqnok{cvx_recur_hol} with $x=x^*$.
\end{proof}

In the next result we specify the convergence rates of the UPFAG algorithm.

\begin{corollary}\label{main_corl_hol}
Let $\{x^{ag}_k\}$ be the iterates generated by Algorithm~\ref{alg_UPFAG}. Suppose there exist some constants
$\lambda >0$ and $\beta >0$ such that
the initial stepsizes $\hat{\lambda}_k \ge \lambda$ and $\hat{\beta}_k \ge \beta$ for all $k \ge 1$.
\begin{itemize}
\item [a)]  Suppose that $\Psi$ is bounded below over $X$, i.e., $\Psi^*$ is finite. Then, for any $N \ge 1$, we have
\beq
\min\limits_{k=1,...,N} \|\hat{g}_{_{X,k}}\|^2
\le \frac{8[\Psi(x_0) - \Psi^*+1]}{\gamma_2} \left[\frac{8 L(\nu,H)}{ (2\sigma - \gamma) N^{\frac{2\nu}{1+\nu}}}+\frac{1}{\beta N}\right].\label{main_nocvx_hol2}
\eeq
\item [b)] Suppose that $f$ is convex and an optimal solution $x^*$ exists for problem \eqnok{NLP}. Then for any $N \ge 1$, we have
\beqa
\Psi(x^{ag}_N)-\Psi(x^*) &\le& \frac{4\|x_0 - x^*\|^2}{\gamma_1 N^{\frac{1+3\nu}{1+\nu}} }\left[\frac{2 L(\nu,H)}{\delta^{\frac{1-\nu}{1+\nu}}}+\frac{1}{\lambda}\right]+\delta. \label{cvx_fun_hol2}
\eeqa
\end{itemize}

\end{corollary}

\begin{proof}
Since $\hat{\lambda}_k \ge \lambda$ and $\hat{\beta}_k \ge \beta$ for all $k \ge 1$, then it follows from
 \eqnok{def_alpha_lambda}, \eqnok{def_beta} , \eqnok{line_search_cond1} and $\eta_k = \alpha_k \lambda_k$ that
\beq \label{line_search_cond2}
\beta_k \ge \gamma_2 \min \left\{\frac{ (2\sigma - \gamma) k^{\frac{\nu-1}{1+\nu}}}{2 L(\nu,H)}, \beta \right\} \ \ \ \mbox{and} \ \ \ \alpha_k^{\frac{2\nu}{1+\nu}} \lambda_k \ge \gamma_1 \min \left\{\frac{ \delta^{\frac{1-\nu}{1+\nu}}}{2 L(\nu,H)}, \lambda \alpha_k^{\frac{\nu-1}{1+\nu}}\right\},
\eeq
which together with
\beqa
\sum_{k=\lfloor N/2 \rfloor+1}^N k^{\frac{1-\nu}{1+\nu}} &\le& \int_{x=\lfloor N/2 \rfloor+1}^{N+1} x^{\frac{1-\nu}{1+\nu}} dx \le  4N^{\frac{2}{1+\nu}},\nn \\
\sum_{k=\lfloor N/2 \rfloor+1}^N \frac{1}{k} &\le& \int_{x=\lfloor N/2 \rfloor}^N \frac{dx}{x}
= \ln \frac{N}{\lfloor N/2 \rfloor} \le 1,
\eeqa
and arithmetic-harmonic mean inequality imply that
\begin{align} \label{ari-har}
&\sum_{k=\lfloor N/2 \rfloor+1}^N  \beta_k \ge \sum_{k=\lfloor N/2 \rfloor+1}^N \gamma_2 \min \left\{\frac{(2 \sigma - \gamma) k^{\frac{\nu-1}{1+\nu}}}{2 L(\nu,H)}, \beta \right\}
\ge \frac{\gamma_2 N^2}{4\sum_{k=\lfloor N/2 \rfloor+1}^N \max \left\{\frac{2 L(\nu,H)} {(2 \sigma - \gamma)k^{\frac{\nu-1}{1+\nu}}}, \frac{1}{\beta}\right\}} \nn \\
&\ge \frac{\gamma_2 N^2}{4\sum_{k=\lfloor N/2 \rfloor+1}^N \left\{2(2 \sigma - \gamma)^{-1} L(\nu,H) k^{\frac{1-\nu}{1+\nu}}+\beta^{-1}\right\}}
\ge \frac{\gamma_2}{4 \left(8 (2\sigma - \gamma)^{-1} L(\nu,H) N^{\frac{-2\nu}{1+\nu}}+(\beta N)^{-1}\right)}.
\end{align}
%
%
Combining the above relation with \eqnok{main_nocvx_hol}, we clearly obtain \eqnok{main_nocvx_hol2}.

Now, observing \eqnok{line_search_cond2} and the facts that $\alpha_k \in (0,1]$  and $\nu \in (0,1]$, we have
\[
\alpha_k^{\frac{2\nu}{1+\nu}} \lambda_k \ge \gamma_1 \min \left\{\frac{ \delta^{\frac{1-\nu}{1+\nu}}}{2 L(\nu,H)}, \lambda \right\},
\]
which together with \eqnok{def_Gamma2} imply that
\[
\lambda_k \ge \left(\frac{\lambda_1}{\Gamma_k}\right)^{\frac{2\nu}{1+3\nu}} \gamma_1^{\frac{1+\nu}{1+3\nu}} \min \left\{\left(\frac{ \delta^{\frac{1-\nu}{1+\nu}}}{2 L(\nu,H)}\right)^{\frac{1+\nu}{1+3\nu}},  \lambda^{\frac{1+\nu}{1+3\nu}}\right\}.
\]
Noticing this observation, defining $c=\frac{1+\nu}{1+3\nu}$, then by
 \eqnok{def_Gamma} and \eqnok{def_Gamma2},  we have
\[
\frac{1}{\Gamma_k^c} -\frac{1}{\Gamma_{k-1}^c} =  \frac{1}{\Gamma_k^c} -\frac{(1-\alpha_k)^c}{\Gamma_k^c} \ge \frac{c \alpha_k}{\Gamma_k^c}
=\frac{c \lambda_k \Gamma_k^{1-c}}{\lambda_1} \ge
\frac{c \gamma_1^c}{\lambda_1^c}
\min \left\{\frac{\delta^{\frac{1-\nu}{1+3\nu}}}{[2 L(\nu,H)]^c},  \lambda^c\right\},
\]
where the first inequality follows from the fact that $1-(1-\alpha)^c \ge c \alpha$ for all $\alpha \in (0,1]$ and $c \in [1/2,1]$.
By the above inequality and noticing $\Gamma_0=1$, similar to \eqnok{ari-har}, for any $N \ge 1$ we obtain
\[
\frac{1}{\Gamma_N^c} \ge \frac{c \gamma_1^c N^2}{\lambda_1^c \sum_{k=1}^N \left\{[2 L(\nu,H)\delta^{\frac{\nu-1}{1+\nu}}]^c+\lambda^{-c}\right\}}
\ge \frac{c \gamma_1^c N}{\lambda_1^c \left\{[2 L(\nu,H)\delta^{\frac{\nu-1}{1+\nu}}]^c +\lambda^{-c}\right\}},
\]
which together with the facts that $c^{\frac{1}{c}} \ge 1/4$ and $(a+b)^{\frac{1}{c}} \le 2 (a^{\frac{1}{c}}+b^{\frac{1}{c}})$ for any $a,b>0$, imply that
\[
\Gamma_N \le \frac{8\lambda_1}{\gamma_1 N^{\frac{1+3\nu}{1+\nu}}}\left[\frac{2 L(\nu,H)}{ \delta^{\frac{1-\nu}{1+\nu}}}+\frac{1}{\lambda }\right].
\]
Combining the above relation with \eqnok{cvx_fun_hol}, clearly we obtain \eqnok{cvx_fun_hol2}.

\end{proof}

We now add a few remarks about the results obtained in Corollary~\ref{main_corl_hol}. First, for simplicity let us assume  $\beta \ge (2\sigma - \gamma)/ L(\nu, H)$. Then, by \eqnok{main_nocvx_hol2}, we conclude that the number of iterations performed by the UPFAG method to have $\|\hat{g}_{_{X,k}}\|^2 \le \epsilon$ for at least one $k$, after disregarding some constants, is bounded by \footnote{This complexity bound was also derived for the gradient descent method as a homework assignment given by the second author 
in Spring 2014, later summarized and refined by one of the class participants in \cite{Yashtini15-1}.
However this development requires the problem to be unconstrained and the paramers $H$ and $\nu$ given a priori.} 
\beq \label{main_nocvx_hol3}
{\cal O} \left(H^{\frac{1}{\nu}} \left[\frac{\Psi(x_0) - \Psi^*}{\epsilon}\right]^{\frac{1+\nu}{2\nu}}\right).
\eeq
Note that when the scaling matrix $G_k$ is also uniformly bounded from above, i.e., $G_k \preceq M I$ for
some constant $M >0 $, $\|\hat{g}_{_{X,k}}\|$ is equivalent to $\|g_{_{X,k}}\|$ defined in  \eqnok{def_proj_grad}.
Hence, when  $\nu=1$ and  $G_k$ is uniformly bounded from above,
the above bound will reduce to the best known iteration complexity in \eqnok{main_nocvx2} for the class of nonconvex functions with Lipschitz continuous gradient.

Second, by choosing $\delta=\epsilon/2$ and assuming that $\lambda \ge \delta^{\frac{1-\nu}{1+\nu}}/L(\nu, H)$, \eqnok{cvx_fun_hol2}
implies that the UPFAG method can find a solution $\bar x$ such that $\Psi(\bar x)-\Psi(x^*) \le \epsilon$, after disregarding some constants,  in at most
\beq \label{cvx_fun_hol3}
{\cal O} \left(\left[\frac{H\|x_0 - x^*\|^{1+\nu}}{\epsilon}\right]^{\frac{2}{1+3\nu}}\right)
\eeq
number of iterations which is optimal for convex programming \cite{nemyud:83}. If $\nu=1$ the above bound will reduce to \eqnok{cvx_fun2} obtained by the UAG method for the class of convex functions with Lipschitz continuous gradient. Note that \eqnok{cvx_fun_hol3} is
on the same order to the bound obtained by the universal fast gradient method proposed by Nesterov~\cite{Nest14} for convex optimization problems.


Finally, it is interesting to find the number of gradient computations at each iteration of Algorithm~\ref{alg_UPFAG}
assuming the initial stepsizes $\hat{\lambda}_k \ge \lambda$ and $\hat{\beta}_k \ge \beta$ for all $k \ge 1$.
According to \eqnok{def_beta} and \eqnok{line_search_cond2}, we conclude that, after disregarding some constants, $\tau_{2,k} \le \log k$ which implies that the number of gradient computations at point $x^{ag}_{k-1}$ is bounded by
$\log k$. Similarly, we obtain that the number of gradient computations at point $x^{md}_{k}$ is also bounded by $\log k$.
Hence, the total number of gradient computations at the $k$-th iteration is bounded by $2 \log k$.
 On the other hand, suppose that we choose $\hat{\beta}_k$ and $\hat{\lambda}_k$ according to \eqnok{def_alpha_lambda_beta}. Then, we have
\[
\tau_{1,k} =  (\log \eta_k-\log \eta_{k-1})/\log \gamma_1 \quad \mbox{ and } \quad
\tau_{2,k}  = (\log \beta_k-\log \beta_{k-1})/\log \gamma_2.
\]
So the number of gradient evaluations in Step1 and Step 2 at the $k$-th iteration is bounded by
\[
1 + \tau_{1,k} = 1 +  (\log \eta_k-\log \eta_{k-1})/\log \gamma_1 \quad \mbox{ and } \quad
1 + \tau_{2,k} = 1+ (\log \beta_k-\log \beta_{k-1})/\log \gamma_2,
\]
which implies the total number of gradient evaluations in Step1 and Step 2 is bounded by
\[
N_\eta = N + \sum_{k=1}^N \tau_{1,k} =  N + \frac{\log \eta_N-\log \eta_0}{\log \gamma_1} \quad \mbox{ and } \quad
N_\beta = N + \sum_{k=1}^N \tau_{2,k} = N + \frac{\log \beta_N-\log \beta_0}{\log \gamma_2}.
\]
Note that \eqnok{line_search_cond2} implies $N_\eta \le N+c_1$ and $N_\beta \le N+c_2 \log N$ for some positive constants $c_1$ and $c_2$.
Hence, the above relations show that  the average number of gradient computations at each iteration is bounded by a constant, which is less than the aforementioned logarithmic bound $\log k$ obtained for the situations where $\hat{\beta}_k$ and $\hat{\lambda}_k$ are chosen according to \eqnok{def_alpha_lambda} and \eqnok{def_beta}. However, in \eqnok{def_alpha_lambda} and \eqnok{def_beta} the algorithm
allows more freedom to choose initial stepsizes.

One possible drawback of the UPFAG method is that we need to fix the accuracy $\delta$ before running the algorithm
and if we want to change it, we should run the algorithm from the beginning. Moreover, we need to implement two line search procedures to find $\beta_k$ and $\lambda_k$ in each iteration.
In the next section, we address these issues by presenting some problem-parameter free bundle-level type algorithms which do not require a fixed target accuracy in advance and only performs one line search procedure in each iteration.

\setcounter{equation}{0}
\section{Unified bundle-level type methods}\label{sec_bundle}
Our goal in this section is to generalize bundle-level type methods, originally designed for convex programming, for solving a class of possiblely
nonconvex nonlinear programming problems. Specifically, in Subsection~\ref{sec_bndl}, we introduce a unified bundle-level type method by incorporating a local search step into an accelerated prox-level method and establish its convergence properties under the boundedness assumption of the feasible set. In Subsection~\ref{sec_fbndl}, we simplify this algorithm and provide its fast variants for solving ball-constrained problems and unconstrained problems
with bounded level sets.

\subsection{Unified accelerated prox-level method}\label{sec_bndl}

In this subsection, we generalize the accelerate prox-level (APL) method presented by Lan~\cite{Lan13-2} to solve a class of nonlinear programming given in \eqnok{NLP}, where $f$ has H\"{o}lder continuous gradient on $X$. Lan~\cite{Lan13-2} showed that the APL method is uniformly optimal for solving problem \eqnok{NLP} when $f$ is convex and satisfies \eqnok{f_holder1}. Here, we combine the framework of this algorithm with a gradient descent step and present a unified accelerated prox-level (UAPL) method for solving both convex and nonconvex optimization.

As the bundle-level type methods, we introduce some basic definitions about the objective function and the feasible set. We first define a function $h(y,x)$ for a given $y \in X$, as
\beq \label{def_h}
h(y,x)=f(y)+\langle f'(y), x-y \rangle +\cX(x) \ \ \mbox{ for any } x \in X.
\eeq
Note that if $f$ is convex, then we have $h(y,x) \le f(x)+\cX(x)=\Psi(x)$ and hence $h(y,x)$ defines a lower bound for $\Psi(x)$. Also, let ${\cal S}_\Psi(l)$ be the level set of $\Psi$ given by ${\cal S}_\Psi(l)=\{x \in X : \Psi(x) \le l\}$ and define a convex compact set $X'$ as a localizer of the level set ${\cal S}_\Psi(l)$ if it satisfies ${\cal S}_\Psi(l) \subseteq X' \subseteq X$. Then, it can be shown \cite{Lan13-2} that, when $\Psi$ is convex, $\min \{l, \underline h(y) \} \le \Psi(x)$ for any $x \in X$, where
\beq \label{def_h2}
\underline h(y) = \min \{h(y,x) : x \in X'\}.
\eeq

Using the above definitions, we present a unified accelerated prox-level (UAPL) algorithm, Algorithm~\ref{alg_UAPL}, for nonlinear programming. We make the following remarks about this algorithm.

First, note that the updating of $\bar x^{ag}_k$ in step 3 of the UAPL algorithm is essentially a gradient descent step, and hence without this update, the UAPL algorithm would reduce to a simple variant of the APL method in \cite{Lan13-1} for convex programming. However, this update is required to establish convergence for the case when the objective function is nonconvex. Second, this UAPL algorithm has two nested loops. The outer loop called phase is counted by index $s$. The number of iterations of the inner loop in each phase is counted by index $t$.
 If we make a progress in reducing the gap between the lower and upper bounds on $\Psi$, we terminate the current phase (inner loop) in step 4 and go to the next one with a new lower bound.
As shown in \cite{Lan13-2}, the number of steps in each phase is finite when $\Psi$ is convex. However, when $\Psi$ is nonconvex, $\underline \Psi$ is not necessary a lower bound on $\Psi$, and hence the termination criterion in Step 4 may not be satisfied. In this case, we could still provide the convergence of the algorithm in terms of the projected gradient defined in \eqnok{def_proj_grad} because of the gradient descent step incorporated in step 3.

The following Lemma due to Lan~\cite{Lan13-2} shows some properties of the UAPL algorithm by generalizing the prox-level method in \cite{BenNem05-1}.

\begin{lemma}\label{lemma_bndl}
Assume that $\Psi$ is convex and bounded below over $X$, i.e., $\Psi^*$ is finite.
Then, the following statements hold for each phase of Algorithm~\ref{alg_UAPL}.
\begin{itemize}
\item [a)]$\{X'_t\}_{t \ge 0}$ is a sequence of localizers of  the level set ${\cal S}_\Psi(l)$.
\item [b)] $\underline \Psi_0 \le \underline \Psi_1 \le \ldots \le \underline \Psi_t \le \Psi^* \le \bar \Psi_t \le \ldots \le \bar \Psi_1 \le \bar \Psi_0$ for any $t \ge 1$.
\item [c)] $\emptyset \neq \underline X_t \subseteq \bar X_t$ for any $t \ge 1$ and hence, Step 5 is always feasible unless the current phase terminates.
\end{itemize}

\end{lemma}

\begin{algorithm} [H]
	\caption{The unified accelerated prox-level (UAPL) algorithm}
	\label{alg_UAPL}
	\begin{algorithmic}

\STATE Input:
$p_0 \in X$, 
$\alpha_t \in (0,1)$ with $\alpha_1=1$, and algorithmic parameter $\eta \in (0,1)$.
\STATE Set $p_1 \in \Argmin_{x \in X} h(p_0,x)$, $\lb_1 = h(p_0,p_1)$, $ x^{ag}_0=p_0$, and $k=0$.

\vgap

{\bf For $s=1,2,\ldots$:}

{\addtolength{\leftskip}{0.2in}
Set $\hat x^{ag}_0=p_s$, $\bar \Psi_0 = \Psi(\hat x^{ag}_0)$, $\underline \Psi_0 =\lb_s$, and $l=\eta \underline \Psi_0 +(1-\eta)\bar \Psi_0$. Also, let $x_0 \in X$ and the initial localizer $X'_0$ be arbitrarily chosen, say $x_0=p_s$ and $X'_0=X$.

\vgap

{\bf For $t=1,2,\ldots$:}

}
\vgap
{\addtolength{\leftskip}{0.4in}

\STATE 1. Update lower bound: set $x^{md}_t=(1-\alpha_t)\hat x^{ag}_{t-1}+\alpha_t x_{t-1}$ and $\underline \Psi_t:= \max \left\{\underline \Psi_{t-1}, \min\{l, \underline h_t\} \right\}$, where $\underline h_t \equiv \underline h(x^{md}_t)$ is defined in \eqnok{def_h2} with $X' = X'_{t-1}$.
\STATE 2. Update the prox-center: set
\beq
x_t = \argmin_{x \in X'_{t-1}} \left\{\|x-x_0\|^2: h(x^{md}_t,x) \le l\right\}. \label{Ne1_bnl}
\eeq
If \eqnok{Ne1_bnl} is infeasible, set $x_t =\hat x^{ag}_{t-1}$.
\STATE 3. Update upper bound: set $k \leftarrow k+1$ and choose $\hat x^{ag}_t$ such that $\Psi(\hat x^{ag}_t) = \min \{\Psi(\hat x^{ag}_{t-1}) , \Psi(\tilde x^{ag}_t), \Psi(\bar x^{ag}_k)\}$, where $\tilde x^{ag}_t=(1-\alpha_t)\hat x^{ag}_{t-1}+\alpha_t x_t$ and $\bar x^{ag}_k$ is obtained by Step 2 of Algorithm~\ref{alg_UPFAG}. Set $\bar \Psi_t= \Psi(\hat x^{ag}_t)$ and $x^{ag}_k =\hat x^{ag}_t$.

\vgap

\STATE 4. Termination: If $\underline  \Psi_t \le \bar \Psi_t$ and $\bar \Psi_t - \underline \Psi_t \le [1-\tfrac{1}{2}\min\{\eta, 1-\eta\}](\bar \Psi_0-\underline \Psi_0)$,
then terminate this phase (loop) and set $p_{s+1} = \hat x^{ag}_t$ and $\lb_{s+1} = \underline \Psi_t$.
\vgap

\STATE 5. Update localizer: choose an arbitrary $X'_t$ such that $\underline X_t \subseteq X'_t \subseteq \bar X_t$, where
\beq \label{update_loc}
\underline X_t = \left\{x \in X'_{t-1} : h(x^{md}_t,x) \le l \right\} \ \ \mbox{and}
\ \ \bar X_t = \left\{x \in X: \langle x_t-x_0, x-x_t \rangle \ge 0\right\}.
\eeq

}
{\addtolength{\leftskip}{0.2in}
{\bf End}

}
{\bf End}

	\end{algorithmic}
\end{algorithm}

\vgap

Now, we can present the main convergence properties of the above UAPL algorithm.
\begin{theorem}\label{main_theom_bndl}
Let the feasible set $X$ be bounded.
\begin{itemize}
\item [a)] Suppose $\Psi$ is bounded below over $X$, i.e., $\Psi^*$ is finite.
The total number of iterations performed by the UAPL method to have $\|g_{_{X,k}}\|^2 \le \epsilon$ for at least one $k$, after disregarding some constants, is bounded by \eqnok{main_nocvx_hol3}.

\item [b)] Suppose that $f$ is convex and
an optimal solution $x^*$ exists for problem \eqnok{NLP}. Then, the number of phases performed by the UAPL method to find a solution $\bar x$ such that $\Psi(\bar x)-\Psi(x^*) \le \epsilon$, is bounded by
\beq\label{phase_bnd}
S(\epsilon) = \left \lceil \max \left\{0, \log_{\frac{1}{q}} \frac{H \max_{x,y \in X} \|x-y\|^{1+\nu}}{(1+\nu)\epsilon}\right\}\right \rceil,
\eeq
where
\beq \label{def_q}
q = 1-\frac{1}{2}\min\{\eta, 1-\eta\}.
\eeq
In addition,  by choosing $\alpha_t = 2/(t+1)$ for any $ t \ge 1$, the total number of iterations to find the aforementioned solution
$\bar x$ is bounded by
\beq\label{iter_bound}
S(\epsilon)+ \frac{1}{1-q^{\frac{2}{3\nu+1}}}\left(\frac{ 4 \sqrt 3 H \max_{x,y \in X} \|x-y\|^{1+\nu}}{\eta \theta (1+\nu) \epsilon}\right)^{\frac{2}{3\nu+1}}.
\eeq

\end{itemize}
\end{theorem}

\begin{proof}
First, note that part a) can be established by essentially following the same arguments as those in
Theorem~\ref{main_theom_hol}.b) and Corollary~\ref{main_corl_hol}.a).
Second, due to the termination criterion in step 4 of Algorithm~\ref{alg_UAPL}, for any phase $s \ge 1$, we have
\[
\Psi(p_{s+1})-\lb_{s+1} \le q [\Psi(p_s)-\lb_s],
\]
which by induction and together with the facts that $\lb_1=h(p_0,p_1)$ and $\lb_s \le \Psi(x^*)$, clearly imply
\beq \label{cvx_fun_bndl}
\Psi(p_s)-\Psi(x^*) \le  \left[\Psi(p_1)-h(p_0,p_1)\right]q^{s-1}.
\eeq
This relation, as shown in Theorem 4 of \cite{Lan13-2} for convex programming, implies \eqnok{phase_bnd}. Third, \eqnok{iter_bound} is followed by \eqnok{phase_bnd} and \cite[Proposition 2, Theorem 3]{Lan13-2}.
\end{proof}

\vgap

We now add a few remarks about the above results. First, note that the UAPL amd UPFAG methods essentially have the same mechanism to ensure the
global convergence when the problem \eqnok{NLP} is  nonconvex.
To the best of our knowledge, this is the first time that a bundle-level type method is proposed for solving a class of possibly
nonconvex nonlinear programming problems.
Second, note that the bound in \eqnok{iter_bound} is in the same order of magnitude as the optimal bound in \eqnok{cvx_fun_hol3} for convex programming. However, to obtain this bound, we need the boundedness assumption on the feasible set $X$, although we do not need the target accuracy a priori. Third, parts a) and c) of Theorem~\ref{main_theom_bndl} imply that the UAPL method can uniformly solve weakly smooth nonconvex and convex problems without requiring any problem parameters. In particular, it achieves the best known convergence rate for nonconvex problems and its convergence rate is optimal if the problem turns out to be convex.

Finally, in steps 1 and 2 of the UAPL algorithm, we need to solve two subproblems which can be time consuming. Moreover, to establish the convergence of this algorithm, we need the boundedness assumption on $X$ as mentioned above. In the next subsection, we address these issues by exploiting the framework of another bundle-level type method which can significantly reduce the iteration cost.

\subsection{Unified fast accelerated prox-level method}\label{sec_fbndl}
In this subsection, we aim to simplify the UAPL method for solving ball-constrained problems and unconstrained problems with bounded level sets.
Recently, Chen at. al~\cite{CheLanOuyZha14} presented a simplified variant of the APL method, namely fast APL (FAPL) method, for ball constrained convex problems. They showed that the number of subproblems in each iteration is reduced from two to one and presented an exact method to solve the subproblem.

In this subsection, we first generalize the FAPL method for ball-constrained nonconvex problems and then discuss how to extend it for unconstrained problems with bounded level sets. It should be mentioned that throughout this subsection, we assume that the simple nonsmooth convex term vanishes in the objective function i.e., $\cX \equiv 0$ in \eqnok{NLP}.

Below we present the unified FAPL (UFAPL) method to sovle the problem \eqnok{NLP} with the ball constraint, i.e., $X = {\cal B}(\bar x, R)$.

\begin{algorithm} [H]
	\caption{The unified fast accelerated prox-level (UFAPL) algorithm}
	\label{alg_UFAPL}
	\begin{algorithmic}

\STATE Input: $p_0 \in {\cal B} (\bar x, R)$ and $\eta, \theta \in (0,1)$.
\STATE Set $p_1 \in \Argmin_{x \in {\cal B} (\bar x, R)} h(p_0,x)$, $\lb_1 = h(p_0,p_1)$, $ x^{ag}_0=p_0$, and $k=0$.

\vgap

{\bf For $s=1,2,\ldots$:}

{\addtolength{\leftskip}{0.2in}

Set $\hat x^{ag}_0=p_s$, $\bar \Psi_0 = \Psi(\hat x^{ag}_0)$, $\underline \Psi_0 =\lb_s$, $l=\eta \underline \Psi_0 +(1-\eta)\bar \Psi_0$, and $X'_0=X=\bbr^n$. Also, let $x_0 \in {\cal B} (\bar x, R)$ be arbitrary given.

\vgap

{\bf For $t=1,2,\ldots$:}

}
\vgap
{\addtolength{\leftskip}{0.4in}

\STATE 1. Set $x^{md}_t=(1-\alpha_t)\hat x^{ag}_{t-1}+\alpha_t x_{t-1}$ and define $\underline X_t$ as in \eqnok{update_loc}.

\STATE 2. Update the prox-center: Let $x_t$ be computed by \eqnok{Ne1_bnl} with $x_0 =\bar x$, i.e.,
$
x_t = \argmin_{x \in \underline X_t} \|x-\bar x\|^2.
$

\STATE 3. Update the lower bound: set $k \leftarrow k+1$ and choose $\hat x^{ag}_t$ such that $\Psi(\hat x^{ag}_t) = \min \{\Psi(\hat x^{ag}_{t-1}) , \Psi(\bar x^{ag}_k)\}$, where $\bar x^{ag}_k$ is obtained by Step 2 of Algorithm~\ref{alg_UPFAG} with $X={\cal B}(\bar x, R)$.
If $\underline X_t = \emptyset$ or $\|x_t-\bar x\| >R$, then terminate this phase with $x^{ag}_k =\hat x^{ag}_t$, $p_{s+1} = \hat x^{ag}_t$, and $\lb_{s+1} = l$.

\vgap

\STATE 4. Update upper bound: let $\tilde x^{ag}_t=(1-\alpha_t)\hat x^{ag}_{t-1}+\alpha_t x_t$. If $\Psi(\tilde x^{ag}_t) < \Psi(\hat x^{ag}_t)$, then set $x^{ag}_k=\hat x^{ag}_t = \tilde x^{ag}_t$ and $\bar \Psi_t = \Psi(\hat x^{ag}_t)$. If $\bar \Psi_t \le l+\theta (\bar \Psi_0-l)$,
then terminate this phase (loop) and set $p_{s+1} = \hat x^{ag}_t$ and $\lb_{s+1} = \lb_s$.
\vgap

\STATE 5. Update localizer: choose any polyhedral $X'_t$ such that $\underline X_t \subseteq X'_t \subseteq \bar X_t$, where $\underline X_t$ and $\bar X_t$ are defined in \eqnok{update_loc} with $X=\bbr^n$.

}
{\addtolength{\leftskip}{0.2in}
{\bf End}

}
{\bf End}

	\end{algorithmic}
\end{algorithm}

We now add a few remarks about the above algorithm. First, note that we do not need to solve the subproblem \eqnok{def_h2} in the UFAPL method.
Moreover, the subproblem \eqnok{Ne1_bnl} in the UFAPL method is to project $\bar{x}$ over a closed polyhedral.
There exist quite efficient methods for performing such a projection on a polyhedral (see e.g.,\cite{CheLanOuyZha14, HagerZhang2015}).
When $G_k = I$, the subproblem associated with finding $\bar x_k^{ag}$ in step 3 of the UFAPL method has a closed-form solution.
Hence, in this case there is only one subproblem to be solved in each iteration of the UFAPL method and this subproblem can be
solved quite efficiently.

Second, note that the UFAPL algorithm can be terminated in step 3 or step 4. Moreover, by combining the convergence results in \cite{CheLanOuyZha14}
and applying similar techniques used in showing the Theorem~\ref{main_theom_hol}, we can establish complexity results similar to the Theorem~\ref{main_theom_bndl} for the UFAPL method.
For simplicity, we do not repeat these arguments here.

Finally, we can extend the above results for the UFAPL method to unconstrained problems. In particular, suppose the level set
\[
{\cal S}_0 = \{x \in \bbr^n \ \ | \Psi(x) \le \Psi(x_0) \},
\]
is bounded, where $x_0$ is the initial point for the UFAPL method. Now, consider the ball-constrained problem
\beq \label{NLP_ball}
\min_{x \in {\cal B} (x_0, R)} \Psi(x),
\eeq
such that $R = \max_{x,y \in {\cal S}_0} \|x-y\| +\delta$ for a given $\delta>0$.

To solve this problem, we could apply the UFAPL method with small modifications. Specifically, we use $X=\bbr^n$ to find $\bar x^{ag}_k$ in step 3 of this method. Now, let $\{x^{ag}_k\}_{k \ge 1}$ be generated by this modified UFAPL method. By steps 3 and 4 of this method, we clearly have $\Psi(x^{ag}_k) \le \Psi(x^{ag}_{k-1})$ for all $k \ge 1$, which implies that $x^{ag}_k \in {\cal S}_0$ for all $k \ge 1$. Hence, all generated points $\{x^{ag}_k\}_{k \ge 1}$ lie in the interior of the aforementioned ball ${\cal B} (x_0, R)$ due to $\delta >0$, which consequently implies that the optimal solution of the problem \eqnok{NLP_ball} is the same as the that of the unconstrained problem $\min_{x \in \bbr^n} \Psi(x)$. Therefore, under the boundedness assumption on
${\cal S}_0$, we can apply the above modified UFAPL method to solve the ball-constrained problem \eqnok{NLP_ball} in order
to solve the original unconstrained problem.
%
%
%
%
%
%
%
%
%
%
%
\setcounter{equation}{0}
\section{Numerical Experiments}\label{sec_num}

In this section, we show the performance of our algorithms for solving the following two problems, namely, the least square problem with nonconvex regularization term and the sigmoid support vector machine (SVM) problem.

\subsection {Nonconvex regularized least square problem}
In our first experiment, we consider the following least square problem with a smoothly clipped absolute deviation
penalty term given in \cite{FanLi01-1}:
\beq \label{SCAD}
\min_{\|x\| \le 1} \Psi(x) := \frac{1}{2}\|Ax-b\|^2 + m \sum_{j=1}^n p_\lambda (|x_j|),
\eeq
where $A \in \bbr^{m \times n}$, $b \in \bbr^m$, the penalty term $p_{\lambda}: \bbr_+ \to \bbr$ satisfies $p_\lambda (0)=0$, and its derivative is given by
\[
p'_\lambda (\beta)= \lambda \left\{I(\beta \le \lambda)+\frac{\max(0, a \lambda-\beta)}
{(a-1)\lambda}I(\beta > \lambda) \right\}
\]
for some constant parameters $a>2$ and $\lambda>0$. Here, $I (\cdot)$ is the indicator function.
As it can be seen, $p_\lambda (|\cdot|)$ is nonconvex and non-differentiable at $0$.
Therefore, we replace $p_{\lambda}$ by its smooth nonconvex approximation $q_\lambda: \bbr_+ \to \bbr$,
satisfying $q_\lambda (0)=0$ and its derivative is defined by
\[
q'_\lambda (\beta)= \left\{\beta I(\beta \le \lambda)+\frac{\max(0, a \lambda-\beta)}{(a-1)}I(\beta > \lambda) \right\}.
\]
Note that \eqnok{SCAD} with the regularization term $p_\lambda$ substituted by
$q_\lambda$ fits the setting of the  problem \eqnok{NLP}, where $\cX \equiv 0$ and $\Psi(x)=f(x)$ has a Lipschitz continuous gradient.

In this experiment, we assume that the elements of $A$ are randomly drawn from the standard normal distribution, $b$ is obtained by
$b = A \bar x +\xi$, where $\xi \sim N(0,\bar{\sigma}^2)$ is the random noise independent of
$A$, and the coefficient $\bar x$ defines the true linear relationship between rows of $A$ and $b$. Also, we set the parameters to $a=3.7$, $\lambda=0.01$, and the noise level to $\bar{\sigma}=0.1$.

We consider three different problem sizes as $n=2000$, $4000$, and $8000$ with $m=1000$, $2000$, and $4000$, respectively. For each problem size, $10$ different instances $(b,A,\bar x,\xi)$ using the aforementioned approach
were generated.
We implement different algorithms including the UAG, UPFAG using partial line search (the stepsizes are set to \eqnok{def_alpha_lambda_beta} with $\hat \lambda_1 = \hat \beta_1 = 1/L$, where $L$ is an estimation for Lipschitz constant of $f'$), UPFAG with full line search showing by UPFAG-full (the stepsizes are set to \eqnok{def_alpha_lambda} and \eqnok{def_beta} with $\hat \lambda_k = \hat \beta_k = 1/L \ \ \forall k \ge 1$), UAPL, UFAPL and the projected gradient method (PG) described after the presentation of Algorithm~\ref{alg_UAG}. Table~\ref{TB1} and Table~\ref{TB2} summarize the average results of this experiment over $10$ instances for each problem size where the initial point is $x_0^{ag}=0$.

\begin{table}
\caption{Average required number of iterations ($Iter(k)$), runtime ($T(s)$), and the best objective value obtained by different algorithms till finding a solution $\bar{x}^*$ satisfying a desired accuracy for the projected gradient over $10$ instances of the nonconvex regularized least square problem.}
\vgap
\centering
\label{TB1}
\footnotesize
\begin{tabular}{|c|c|c|c|c|c|c|c|}
\hline
\multicolumn{2}{|c|}{\multirow{2}{*}{$\|g_{_X} (\bar{x}^*)\|^2$}}&&&&&&\\
\multicolumn{2}{|c|}{}&$<10^0$&$<10^{-1}$&$<10^{-2}$&$<10^{-3}$&$<10^{-4}$&$<10^{-5}$\\
\hline
\multicolumn{8}{|c|}{$m=1000, n=2000$}\\
\hline
\multirow{3}{*}{PG}&Iter($k$)&1645 & 3011 &$>3324$ & & & \\
&$T(s)$& 11.7 &21.3 &$>23.7$& & & \\
&$\Psi(x_k^{ag})$&5.76e+01 &5.73e+01  &5.73e+01& 5.73e+01& 5.73e+01& 5.73e+01\\
\hline
\multirow{3}{*}{UAG}&Iter($k$)&368& 540& 607& 652& 752& 816   \\
&$T(s)$& 5.2 &7.6 &8.6 &9.2& 10.7& 11.6   \\
&$\Psi(x_k^{ag})$&5.73e+01 &5.71e+01 &5.71e+01& 5.71e+01 &5.71e+01& 5.71e+01   \\
\hline
\multirow{3}{*}{UPFAG}&Iter($k$)& 323& 414& 483& 659& 767& 831     \\
&$T(s)$& 4.6 &5.9& 6.9& 9.5& 11.1& 12.0    \\
&$\Psi(x_k^{ag})$&5.74e+01 &5.73e+01 &5.73e+01& 5.73e+01 &5.72e+01& 5.72e+01 \\
\hline
\multirow{3}{*}{UPFAG-full}&Iter($k$)&323& 414& 483& 659& 767& 831\\
&$T(s)$&4.6 &6.0& 6.9& 9.6& 11.2& 12.1   \\
&$\Psi(x_k^{ag})$&5.74e+01 &5.73e+01 &5.73e+01& 5.73e+01 &5.72e+01& 5.72e+01 \\
\hline
\multirow{3}{*}{UAPL}&Iter($k$)&214& 275& 332& 414& 455& 520   \\
&$T(s)$&52.9& 68.7 &83.2& 103.9 &114.3& 130.7     \\
&$\Psi(x_k^{ag})$& 5.75e+01 &5.74e+01 &5.74e+01& 5.74e+01 &5.74e+01& 5.74e+01\\
\hline
\multirow{3}{*}{UFAPL}&Iter($k$)& 209& 312& 361& 426& 472& 517
 \\
&$T(s)$& 4.4 &6.6 &7.6 &8.8 &9.8 &10.7    \\
&$\Psi(x_k^{ag})$&5.72e+01 &5.71e+01 &5.71e+01& 5.71e+01 &5.71e+01& 5.71e+01\\
%
%
%
\hline
\multicolumn{8}{|c|}{$m=2000, n=4000$}\\
\hline
\multirow{3}{*}{PG}&Iter($k$)&2030 &5340& $>7000$ & & & \\
&$T(s)$& 64.4 & 170.2 &$>225.7$& & & \\
&$\Psi(x_k^{ag})$&1.60e+02 &1.60e+02 &1.60e+02& 1.60e+02& 1.60e+02& 1.60e+02\\
\hline
\multirow{3}{*}{UAG}&Iter($k$)&415& 529& 683& 897& 1190& 1527   \\
&$T(s)$& 26.3 &33.7 &43.3 &57.0 &75.7 &96.9\\
&$\Psi(x_k^{ag})$&1.60e+02 &1.60e+02 &1.60e+02& 1.60e+02& 1.60e+02& 1.60e+02\\
\hline
\multirow{3}{*}{UPFAG}&Iter($k$)& 434& 535& 751& 962& 1253& 1591   \\
&$T(s)$& 27.8 &34.1& 47.4& 60.8& 79.4& 101.0   \\
&$\Psi(x_k^{ag})$&1.60e+02 &1.60e+02 &1.60e+02& 1.60e+02& 1.60e+02& 1.60e+02\\
\hline
\multirow{3}{*}{UPFAG-full}&Iter($k$)&434& 535& 751& 962& 1253& 1591\\
&$T(s)$&28.0 &34.4 &48.2 &61.7 &80.1& 101.4   \\
&$\Psi(x_k^{ag})$&1.60e+02 &1.60e+02 &1.60e+02& 1.60e+02& 1.60e+02& 1.60e+02\\
\hline
\multirow{3}{*}{UAPL}&Iter($k$)&234& 404& 514& 672& 858& 1013  \\
&$T(s)$&152.8 &264.6& 337.0& 441.0& 564.2& 666.6   \\
&$\Psi(x_k^{ag})$& 1.60e+02 &1.60e+02 &1.60e+02& 1.60e+02& 1.60e+02& 1.60e+02\\
\hline
\multirow{3}{*}{UFAPL}&Iter($k$)& 242& 354& 490& 547& 596& 767   \\
&$T(s)$& 20.5 &29.7 &41.1& 45.8& 49.9& 63.7   \\
&$\Psi(x_k^{ag})$&1.60e+02 &1.60e+02 &1.60e+02& 1.60e+02& 1.60e+02& 1.60e+02\\
\hline
\end{tabular}
\end{table}

\begin{table}
\caption{Average required number of iterations ($Iter(k)$), runtime ($T(s)$), and the best objective value obtained by different algorithms till finding a solution $\bar{x}^*$ satisfying a desires accuracy for the projected gradient over $10$ instances of the nonconvex regularized least square problem with $m=4000, n=8000$.}
\vgap
\centering
\label{TB2}
\footnotesize
\begin{tabular}{|c|c|c|c|c|c|c|c|}
\hline
\multicolumn{2}{|c|}{\multirow{2}{*}{$\|g_{_X} (\bar{x}^*)\|^2$}}&&&&&&\\
\multicolumn{2}{|c|}{}&$<10^0$&$<10^{-1}$&$<10^{-2}$&$<10^{-3}$&$<10^{-4}$&$<10^{-5}$\\
\hline
\multirow{3}{*}{PG}&Iter($k$)&1364& 2742& $>3471$&  & & \\
&$T(s)$& 169.5 &341.1 &$>430.6$ & & & \\
&$\Psi(x_k^{ag})$&3.79e+02& 3.79e+02&3.79e+02& 3.79e+02&3.79e+02 &3.79e+02   \\
\hline
\multirow{3}{*}{UAG}&Iter($k$)&208& 299& 487& 757& 1184& 1689    \\
&$T(s)$& 50.9 &73.6& 119.8& 186.8& 292.5& 416.6   \\
&$\Psi(x_k^{ag})$&3.79e+02& 3.79e+02&3.79e+02& 3.79e+02&3.79e+02 &3.79e+02\\
\hline
\multirow{3}{*}{UPFAG}&Iter($k$)& 204& 297& 483& 767& 1179& 1693\\
&$T(s)$& 50.0 &73.0 &119.3& 188.7& 289.9& 416.4   \\
&$\Psi(x_k^{ag})$&3.79e+02& 3.79e+02&3.79e+02& 3.79e+02&3.79e+02 &3.79e+02\\
\hline
\multirow{3}{*}{UPFAG-full}&Iter($k$)&204& 297& 483& 767& 1179& 1693\\
&$T(s)$&50.7 &73.6 &119.7& 190.2& 291.3& 418.1  \\
&$\Psi(x_k^{ag})$&3.79e+02& 3.79e+02&3.79e+02& 3.79e+02&3.79e+02 &3.79e+02\\
\hline
\multirow{3}{*}{UAPL}&Iter($k$)&146& 203& 339& 434& 598& 721   \\
&$T(s)$&322.1 &448.8& 754.2& 967.2& 1332.5& 1611.4  \\
&$\Psi(x_k^{ag})$& 3.79e+02& 3.79e+02&3.79e+02& 3.79e+02&3.79e+02 &3.79e+02\\
\hline
\multirow{3}{*}{UFAPL}&Iter($k$)& 124 &198& 309& 362& 561 &677   \\
&$T(s)$& 38.8 &62.1& 97.0& 113.7& 175.8& 212.1   \\
&$\Psi(x_k^{ag})$&3.79e+02& 3.79e+02&3.79e+02& 3.79e+02&3.79e+02 &3.79e+02\\
\hline
\end{tabular}
\end{table}

The following observations can be made from the numerical results.
First, the projected gradient method performs the worst among all the compared algorithms. Second, the performances of UAG and
the variants of UPFAG methods are similar to each other and the type of line search does not have a significant affect on the results. One possible reason might be that only a few number of line searches are needed
when solving this problem.
 Third, while the UAPL and UFAPL methods perform similarly in terms of the iterations and objective value, the former takes significantly longer CPU time as expected due to the existence of two subproblems at each iteration. Finally, the bundle-level type methods outperform the accelerated gradient methods. This observation has been also made for convex problems (see e.g., \cite{ CheLanOuyZha14, Lan13-1}).

\subsection {Nonconvex support vector machine problem} In our second experiment,
we consider a support vector machine problem with nonconvex sigmoid loss function \cite{MasBaxBarFre99}, that is
\beq \label {sigmoid_loss}
\min_{\|x\| \le a} \left\{\Psi(x) :=\frac{1}{m}\sum_{i=1}^m \left[1 - \tanh \left( v_i \langle x, u_i \rangle \right) \right] + \frac{ \lambda}{2} \|x\|_2^2\right\},
\eeq
for some $\lambda>0$. Note that the first term in the objective function is smooth, nonconvex, and has  Lipschitz continuous gradient. Hence, this problem also fits into the setting of problem \eqnok{NLP} with $\cX \equiv 0$.
Here, we assume that each data point $(u_i, v_i)$ is drawn from the uniform distribution on
$[0,1]^n \times \{-1,1\}$, where $u_i \in \bbr^n$ is the feature vector and $v_i \in \{-1,1\}$ denotes the corresponding
label. Moreover, we assume that $u_i$ is sparse with $5\%$ nonzero components and
$v_i = \text{sign} \left(\langle \bar{x} , u_i \rangle \right)$ for some $\bar{x} \in \bbr^n$ with $\|\bar x\| \le a$.
In this experiment, we set $\lambda=0.01$, $a=50$, and consider two different problem sizes as
$n=2000$ and $4000$ with $m=1000$ and $2000$, respectively. The initial point is randomly chosen within the
ball centered at origin with radius $a$.
Similar to the first experiment, we report the average results of running different algorithms over $10$ instances for each problem size.
Moreover, in order to further assess the quality of the generated solutions, we also
report the classification error evaluated at the classifier $x$ given by
\beq \label{mis_error}
er(x):=\frac{|\left \{i: v_i \neq \text{sign} (\langle x , u_i \rangle), i = 1, ..., K \right \}|}{K},
\eeq
where $K=10000$ for each problem instance.
Table~\ref{TB3} summarizes the results of this experiment over $10$ instances for each problem size.

Similar to the previous experiment, it can be seen that the PG method again performs the worst among all the compared
algorithms and the accelerated prox-level methods also outperform the other algorithms. The UFAPL has the best performance in terms of the iteration, runtime and the classification error.

\begin{table}
\caption{Average required number of iterations (Iter(k)), runtime ($T(s)$), objective value, and classification error found till reaching a desired accuracy for $\|g_{_X} (\bar{x}^*)\|^2 $ over $10$ instances of the sigmoid SVM problem.}
\vgap
\centering
\label{TB3}
\footnotesize
\begin{tabular}{|c|c|c|c|c|c|c|c|}
\hline
\multicolumn{2}{|c|}{\multirow{2}{*}{$\|g_{_X} (\bar{x}^*)\|^2$}}&&&&&&\\
\multicolumn{2}{|c|}{}&$<10^0$&$<10^{-1}$&$<10^{-2}$&$<10^{-3}$&$<10^{-4}$&$<10^{-5}$\\
\hline
\multicolumn{8}{|c|}{$m=1000, n=2000$}\\
\hline
\multirow{4}{*}{PG}&Iter($k$)&1& 6& 805& 1646& 2223& 2939  \\
&$T(s)$& 0.0 &0.0& 3.1& 6.2& 8.4 &11.1  \\
&$\Psi(x_k^{ag})$&4.27e+0& 4.20e+0& 7.71e-1& 2.10e-1& 1.76e-1& 1.70e-1  \\
&$er(x_k^{ag})$&47.53& 46.22& 40.50& 32.68& 31.17& 30.96 \\
\hline
\multirow{4}{*}{UAG}&Iter($k$)&1&5& 94& 195& 311& 559  \\
&$T(s)$& 0.0 &0.1& 0.7 &1.5& 2.4& 4.2   \\
&$\Psi(x_k^{ag})$&4.27e+0& 4.19e+0 &4.17e-1& 2.09e-01& 1.98e-01& 1.67e-01   \\
&$er(x_k^{ag})$& 47.53 &45.91& 32.78& 31.41 &30.99& 30.86   \\

\hline
\multirow{4}{*}{UPFAG}&Iter($k$)& 1& 4& 92& 192& 314& 583   \\
&$T(s)$& 0.0 &0.0 &0.7 &1.5& 2.5& 5.2 \\
&$\Psi(x_k^{ag})$&4.27e+0& 4.19e+0& 4.19e-01 &2.10e-01& 1.99e-01& 1.97e-01    \\
&$er(x_k^{ag})$&47.53 &45.94 &32.75& 31.42& 30.99& 30.95   \\
\hline
\multirow{4}{*}{UPFAG-full}&Iter($k$)&1& 4& 92& 192& 314& 583\\
&$T(s)$&0.0 &0.0& 0.7& 1.5& 3.2& 13.8    \\
&$\Psi(x_k^{ag})$&4.27e+0& 4.19e+0& 4.19e-01 &2.10e-01& 1.99e-01& 1.97e-01    \\
&$er(x_k^{ag})$&47.53 &45.94 &32.75& 31.42& 30.99& 30.95   \\
\hline
\multirow{4}{*}{UAPL}&Iter($k$)&1& 5& 12& 40& 61& 95   \\
&$T(s)$&0.2& 0.9 &2.3 &8.6 &13.6& 21.4       \\
&$\Psi(x_k^{ag})$& 1.04e+01& 3.68e+0& 7.53e-01& 2.56e-01 &2.44e-01 &2.43e-01   \\
&$er(x_k^{ag})$& 43.51& 42.70 &39.28& 31.51 &31.22& 31.23   \\
\hline
\multirow{4}{*}{UFAPL}&Iter($k$)& 1& 3& 27& 42& 64& 100    \\
&$T(s)$& 0.0 &0.1 &0.5& 0.7 &1.0 &1.6     \\
&$\Psi(x_k^{ag})$&4.27e+0& 3.99e+0& 4.01e-01 &2.17e-01 &2.00e-01 &1.98e-01   \\
&$er(x_k^{ag})$&47.53 &45.04& 35.04& 31.59 &31.09& 30.97   \\
\hline
%
%
%
\hline
\multicolumn{8}{|c|}{$m=2000, n=4000$}\\
\hline
\multirow{4}{*}{PG}&Iter ($k$)&1& 47& 1648& 3389& 4853& 6324    \\
&$T(s)$& 0.0& 0.9& 31.2& 64.2& 91.9& 119.5     \\
&$\Psi(x_k^{ag})$&5.57e+0& 4.93e+0& 8.41e-01 &2.48e-01& 2.07e-01& 2.01e-01   \\
&$er(x_k^{ag})$&53.61 &36.17& 31.22& 26.39 &25.29& 25.28    \\
\hline
\multirow{4}{*}{UAG}&Iter ($k$)&1& 15& 121& 232& 356& 584   \\
&$T(s)$& 0.0 &0.6& 4.6& 8.9& 13.6& 22.2    \\
&$\Psi(x_k^{ag})$&5.57e+0& 4.93e+0& 6.62e-01& 3.15e-01& 2.91e-01& 2.90e-01    \\
&$er(x_k^{ag})$& 53.61 &35.51& 27.46& 26.15 &26.02& 25.89    \\

\hline
\multirow{4}{*}{UPFAG}&Iter ($k$)& 1& 14& 118& 237 &374 &648\\
&$T(s)$&0.0 &0.5& 4.5& 9.2& 14.4& 27.4    \\
&$\Psi(x_k^{ag})$&5.57e+0& 4.94e+0& 6.69e-01& 3.17e-01& 2.92e-01 &2.90e-01       \\
&$er(x_k^{ag})$&53.61 &35.59& 27.45 &26.14& 26.00& 25.91    \\
\hline
\multirow{4}{*}{UPFAG-full}&Iter ($k$)&1& 14& 118& 235& 360& 635   \\
&$T(s)$&0.0 &0.5& 4.5& 9.5& 14.5& 52.0    \\
&$\Psi(x_k^{ag})$&5.57e+0& 4.94e+0& 6.69e-01& 3.15e-01& 2.91e-01 &2.90e-01      \\
&$er(x_k^{ag})$&53.61 &35.59& 27.45 &26.14& 26.00& 25.91    \\
\hline
\multirow{4}{*}{UAPL}&Iter ($k$)& 1& 5& 11& 26& 41& 98  \\
&$T(s)$&0.6 &2.8& 6.5& 15.8& 25.2& 62.3         \\
&$\Psi(x_k^{ag})$& 9.60e+0& 2.88e+0& 5.98e-01& 3.63e-01& 3.44e-01& 3.43e-01      \\
&$er(x_k^{ag})$&32.44 &32.41& 28.54& 26.48& 26.23& 26.35   \\
\hline
\multirow{4}{*}{UFAPL}&Iter ($k$)& 1& 4& 27& 45& 66& 92
 \\
&$T(s)$& 0.1 &0.3& 1.5& 2.5& 3.7& 5.1     \\
&$\Psi(x_k^{ag})$&5.57e+0& 4.76e+0& 4.77e-01& 2.65e-01& 2.36e-01 &2.34e-01   \\
&$er(x_k^{ag})$&53.61 &34.74& 26.87& 25.90& 25.33& 25.30     \\
\hline
\end{tabular}
\end{table}

\section{Concluding Remarks}
In this paper, we extend the framework of uniformly optimal algorithms, currently designed for convex programming, to nonconvex nonlinear programming.
In particular, by incorporating a gradient descent step into the framework of uniformly optimal convex programming methods, namely, accelerated gradient and bundle-level type methods, and enforcing the function values evaluated at each iteration of these methods non-increasing, we present unified algorithms for minimizing composite objective functions given by summation of a function $f$ with Holder continuous gradient and simple convex term over a convex set. We show that these algorithms exhibit the best known convergence rate when $f$ is nonconvex and possess the optimal convergence rate if $f$ turns out to be convex. Therefore, these algorithms allow us to have a unified treatment for nonlinear programming problems regardless of their smoothness level and convexity property.
Furthermore, we show that the gradient descent step can be replaced by some Quasi-Newton steps to possibly improve the practical performance of these algorithms. Some numerical experiments are also presented to show the performance of our developed algorithms for
solving a couple of nonlinear programming problems in data analysis.

\bibliographystyle{plain}
\bibliography{../glan-bib}
\end{document}